\documentclass[11pt]{article}
\usepackage{amsmath}
\usepackage{epsfig,amssymb,latexsym,comment}
\usepackage[all]{xy}
\usepackage{setspace}

\numberwithin{equation}{section}

\def\q{\quad}
\def\qed{\unskip\nobreak\hfill$\Box$\par\addvspace{\medskipamount}}

\def\s{{\sigma}}

\def\p{{\rho}}
\def\td{\mathrm{d}}
\newcommand{\be}{\begin{equation}}
\newcommand{\ee}{\end{equation}}
\newcommand{\bea}{\begin{eqnarray}}
\newcommand{\eea}{\end{eqnarray}}
\newcommand{\beas}{\begin{eqnarray*}}
\newcommand{\eeas}{\end{eqnarray*}}

\newtheorem{theorem}{Theorem}[section]
\newtheorem{definition}[theorem]{Definition}
\newtheorem{Prop}[theorem]{Proposition}
\newtheorem{proposition}[theorem]{Proposition}

\newtheorem{lemma}[theorem]{Lemma}
\newtheorem{Lemma}[theorem]{Lemma}

\newtheorem{as}[theorem]{Assumption}
\newtheorem{remark}[theorem]{Remark}
\newtheorem{example}[theorem]{Example}
\newtheorem{examples}[theorem]{Example}
\newtheorem{foo}[theorem]{Remarks}
\newenvironment{Example}{\begin{example}\rm}{\end{example}}

\newenvironment{Remark}{\begin{remark}\rm}{\end{remark}}

\newenvironment{As}{\begin{as}\rm}{\end{as}}

\newenvironment{proof}[1]{\addvspace{\medskipamount}\par\noindent{\it Proof#1}.}
{\unskip\nobreak\hfill$\Box$\par\addvspace{\medskipamount}}

\newcommand{\IR}{\int_{\mathbb{R}^k\setminus\{0\}}}

\newcommand{\E}[1]{{\mathbb{E}}\left[#1\right]}

\DeclareMathOperator{\esssup}{ess\,sup}

\def\F{\mathcal{F}}
\def\Q{\mathcal{Q}}
\def\A{\mathcal{A}}
\def\R{\mathbb{R}}
\def\S{\mathcal{S}}

\begin{document}

\title{{\bf\large On Dynamic Deviation Measures and Continuous-Time Portfolio Optimisation}}
\author{Martijn Pistorius\footnote{Department of Mathematics, Imperial College London, 
m.pistorius@imperial.ac.uk}\qquad\qquad Mitja Stadje\footnote{{Faculty of Mathematics and Economics,
University of Ulm}, mitja.stadje@uni-ulm.de
\newline
\newline
{\em (2010) AMS Classification.} 60H30, 90C46, 91A10, 91B70, 93E99.
}
}

\date{}

\maketitle
\begin{spacing}{0.85}
\begin{quote}
{\small{\bf Abstract.} In this paper we propose the notion of {\em dynamic deviation measure}, 
as a dynamic time-consistent extension of the (static) notion of deviation measure. 
To achieve time-consistency we require that a dynamic deviation measures satisfies
a generalised conditional variance formula. We show that, under a domination condition,    
dynamic deviation measures are characterised as the solutions to a certain class of 
backward SDEs. We  establish for any dynamic deviation measure an integral representation, and 
derive a dual characterisation result  in terms of {\em additively $m$-stable} dual sets. 
Using this notion of dynamic deviation measure 
we formulate a dynamic mean-deviation portfolio optimisation problem 
in a jump-diffusion setting and identify a subgame-perfect Nash equilibrium strategy that is linear 
as function of wealth by deriving and solving an associated extended HJB equation.
}
\end{quote}
\end{spacing}
\section{Introduction}
One traditional way of thinking about risk is in terms of the extend that random realisations 
deviate from the mean.  In portfolio theory as initiated in Markowitz~(1952), for instance, risk is quantified as the variance or standard deviation of the return. In the setting of the Black-Scholes~(1973) model, it is the volatility
parameter, which is equal to the standard deviation of the log-stock price at unit time, 
that is often taken as description
of the risk.  Alternative approaches to quantification of risk that have emerged more recently 
also take into account other aspects of the return distribution such as heavy tails and asymmetry. 
In this context an axiomatic framework for (general) deviation measures 
was introduced and developed in Rockafellar {\em et al.}~(2006a), which form a certain class of non-negative positively homogeneous (static) operators acting on square-integrable random variables. General deviation measures allow to distinguish between upper and lower deviations from the mean, generalising standard deviation. Various aspects of portfolio optimisation and financial decision making
under general deviation measures have been explored in the literature, in particular regarding CAPM, asset betas, one- and 
two-fund theorems and equilibrium theory; see also among many others Cheng {\em et al.}~(2004), Rockafellar {\em et
al.}~(2006b,\,2006c,\,2007),  M\"arket and Schultz~(2005), Stoyanov {\em et al.}~(2008), Grechuk {\em et al.}~(2009), 
or Grechuk and Zabarankin (2013,\,2014). 
In this paper we present an axiomatic approach to deviation measures in 
\emph{dynamic} continuous-time settings. We show that such dynamic deviation measures  
admit in general a dual robust representation and are linked to a certain 
family of backward stochastic differential equations (BSDEs), if a certain domination conditon is satisfied. 
Subsequently, we use this notion of dynamic deviation measure to phrase a mean-deviation portfolio optimisation   
problem in a jump-diffusion setting and identify for this problem a
subgame-perfect Nash equilibrium portfolio allocation strategy by means of an associated novel type of extended Hamilton-Jacobi-Bellman equation, which complements the ones studied in Bj\"{o}rk and Murgoci (2010). 
\smallskip

\noindent{\bf (Conditional) deviation measures.} Dynamic deviation measures are given in terms of 
conditional deviation measures, which are in turn a conditional version of the notion of (static) deviation measure defined in Rockafellar {\em et al.}~(2006a) that we describe next.  On a filtered probability space $(\Omega,\F,(\F_{t})_{t\in [0,T]},\mathbb P)$, where $T>0$ denotes the horizon, consider the (risky) positions described by elements in $L^{p}(\F_t)$, $t\in[0,T]$, $p\ge0$, the space of $\F_t$-measurable random variables $X$ such that $\E{|X|^p}< \infty$); by $L^{p}_{+}(\F_t)$, $L^{\infty}(\F_t)$ and $L^{\infty}_+(\F_t)$ are denoted the subsets of non-negative, bounded and non-negative bounded elements in $L^{p}(\F_t)$. The definition is given as follows:
\begin{definition}
For any given $t\in [0,T]$, $D_t:L^2(\F_T)\to L^2_+(\F_t)$ is called 
an {\em $\mathcal F_t$-conditional deviation measure}
if it is {\em normalised} ($D_t(0) = 0$)
and the following
properties are satisfied:
\begin{itemize}
\item[{\rm(D1)}]
{\em Translation Invariance}: $D_t(X + m) = D_t(X)$ for any  $m\in
L^\infty(\F_t)$;
\item[{\rm(D2)}] {\em Positive Homogeneity}: $D_t(\lambda X) = \lambda D_t(X)$ for any $X \in L^2(\F_T)$ and $\lambda \in L^\infty_+(\F_t)$;
\item[{\rm(D3)}] {\em Subadditivity}: $D_t(X + Y) \leq D_t(X) + D_t(Y)$ for any $X, Y \in L^2(\mathcal F_T)$;
\item[{\rm(D4)}] {\em Positivity}: $D_t(X) \geq 0$ for any $X\in L^2(\mathcal F_T),$ and $D_t(X) = 0$ if and only if $X$ is $\F_t$-measurable.
\end{itemize}
\end{definition} 
If $\mathcal F_0$ is trivial, $D_0$ is a deviation measure in the sense of Definition~1 in 
Rockafellar {\em et al.}~(2006a). The value $D_t(X)=0$, we recall, corresponds to the riskless state of no uncertainty, and axiom (D1) can be interpreted as the requirement that adding to a position $X$ a constant (interpreted as cash)  should not increase the risk. Furthermore, it follows similarly as in Rockafellar {\em et al.} (2006a) 
that, if $D$ satisfies (D2)--(D3), 
(D1) holds if and only if $D_t(m)=0$ for any $m\in L^2(\F_t)$.  In other words, constants do not carry any risk.
Moreover, it is well known that if (D2) holds, (D3) is equivalent to \emph{conditional convexity}, that is,
for any $X,Y\in L^2(\F_T)$ and 
any $\lambda\in L^\infty(\F_t)$ that is such that $0\leq\lambda\leq 1$ 
$$D_{t}(\lambda X+(1-\lambda)Y)\leq \lambda D_{t}(X)+(1-\lambda)D_{t}(Y). $$
The property of convexity is often given the interpretation that
diversification of a position should not increase its riskiness. 
We also note that (D2) implies 
that, for any $X_1, X_2\in L^2(\F_T)$, $D_t(I_AX_i) = I_AD_t(X_i)$, $i=1,2$, 
where $I_A$ denotes the indicator of the set $A$, so that\footnote{To see that \eqref{localproperty} holds note that by (D2) 
$I_AD_t(1_A X_1 + 1_{A^c} X_2)= D_t (I_A(I_AX_1 + 1_{A^c}X_2) = D_t(I_A X_1)=I_A D_t(X_1)$; 
similarly, we have $I_{A^c}D_t(1_A X_1 + 1_{A^c} X_2) = I_{A^c} D_t(X_2).$}
\be
\label{localproperty}
D_t(I_A
X_1+I_{A^c}X_2)=I_AD_t(X_1)+I_{A^c}D_t(X_2), \qquad A\in \F_t.
\ee
In the analysis typically also a lower semi-continuity condition is imposed, the conditional version of which is given as follows:
\begin{itemize}
\item[{\rm (D5)}] {Lower Semi-Continuity:} 
{\it If $X^n$ converges to $X$ in $L^2(\F_T)$ then
$  D_t(X)\leq \liminf_{n}D_t(X^n)$. }
\end{itemize}

\noindent{\bf Dynamic deviation measures.} 
We impose additional structure 
on a given family of $\mathcal F_t$-conditional deviation measures 
in order to ensure it satisfies a form of time-consistency.
One recursive structure that has been succesfully deployed 
in among others the case of mean-variance portfolio optimisation is the one embedded in the {\em conditional variance formula}; 
see for instance Basak and Chabakauri (2010), Wang and Forsyth (2011), Li {\em et al.} (2012) or Czichowsky (2013). 
Inspired by this recursive structure
we require that a collection $(D_t)_{t\in [0,T]}$ of conditional deviation measures satisfy the following generalisation 
of the {conditional variance formula}: 
\begin{itemize}
\item[(D6)] {Time-Consistency:} {\it For all $s,t\in [0,T]$ with $s\leq
t$ and $X\in L^2(\mathcal F_T)$}
\begin{equation}\label{eq:gendev}
D_{s}(X)=D_s(\E{X|\F_t})+\E{D_t(X)|\F_s}.
\end{equation}
\end{itemize}
\begin{Remark}\label{rem1} 
{\bf (i)} As $D(X)\geq 0$, (D6) implies that $(D_s(X))_{s\in[0,T]}$ is a
supermartingale, which implies in particular that $D$ has a c\`adl\`ag
modification. 

\noindent{\bf (ii)} It follows by standard arguments that (D6) for $s=0$ already uniquely determines 
a dynamic deviation measure $D$. 
For suppose that $D_0$ and 
$X\in L^2(\F_T)$ are given and besides $(D_t(X))_{t\in[0,T]}$ there exists a collection of square-integrable
$\F_s$-measurable random variables $(D'_t(X))_{t\in[0,T]}$ satisfying (D6) for $s=0$, then $D_t(X) = D'_t(X)$ for all 
$t\in[0,T]$. Indeed, if the $\F_t$-measurable set $A':=\{D'_t(X)> D_t(X)\}$ were to have  
non-zero measure, then by (\ref{localproperty}) and (D6) we find
$$\E{I_{A'} D_t(X)}=\E{D_t(I_{A'} X)}=D_0(I_{A'} X)-D_0(\E{I_{A'} X|\F_t})
=\E{D'_t(I_{A'} X)}=\E{I_{A'} D'_t(X)},$$
which is a contradiction to the definition of the set $A'$. Similarly, it may be seen that the
set $\{D'_t(X)< D_t(X)\}$ has measure zero.

\noindent{\bf (iii)} Since $D_0$ is convex, lower semi-continuous and finite, $D_0$ is continuous in $L^2(\F_T)$ (see Proposition~2 in Rockafellar {\em et al.} (2006)).	 
\end{Remark}

\noindent We arrive thus at the following definition of dynamic deviation measure:

\begin{definition}
A family $(D_t)_{t\in [0,T]}$ is called a {\em dynamic deviation measure} if 
$D_t$, $t\in [0,T]$, are {$\mathcal F_t$-conditional  deviation measures} 
satisfying (D5) and (D6).
\end{definition} 
One way to construct examples of dynamic deviation measures is in terms of the solutions of a certain type of BSDEs. Such solutions, when seen as function of the  corresponding random variable, we will call $g$-deviation measures (where $g$ is the driver function of the BSDE in question). We show in Theorem~\ref{Main} that, under a domination condition, any dynamic deviation measure is equal to a $g$-deviation measure for some driver function $g$. This result may be considered to be an analogue of the link between the dynamic coherent and convex risk measures and $g$-expectations; see Coquet {\em et al.}~(2002) and Royer~(2006) (for contributions on convex risk measures and $g$-expectations and their generalizations see for instance Barrieu and El Karoui~(2005,2009), Rosazza Gianin~(2006), Kl\"oppel and Schweizer~(2007), Jiang~(2008), El Karoui and Ravenelli~(2009), Bion-Nadal and Magali (2012) or Pelsser and Stadje~(2014)). By drawing on dual robust representation results we also establish characterisations of general dynamic deviation measures that are valid without the domination condition (see Theorems~\ref{Nondom}, \ref{mstable}  and \ref{theoremdual1}).

\begin{remark}[Relation to dynamic coherent risk-measures]{\rm
By generalising arguments given in Rockafellar {\em et al.}~(2006) to the $\mathcal F_t$-conditional context, we note that 
any $\mathcal F_t$-conditional deviation measure is equal to the sum of a conditional expectation and a risk-measure
$\p_t$ that satisfies a ($\mathcal F_t$-conditional) lower range dominance condition
(that is,  $\rho_t(X) \ge \E{X|\mathcal F_t}$ for all $t\in[0,T]$ and $X\in L^2(\F_T)$ with equality 
on sets in $\mathcal F_t$ on which $X$ is constant). 
As the notions of time-consistency differ in cases of dynamic deviation and dynamic risk measures
this relation does not carry over to the dynamic case. A collection $(\p_t)_{t\in [0,T]}$, $\p_t:L^2(\F_T)\to
L^2_+(\F_t)$, forms a family of dynamic coherent risk measures, we recall, if, for every $t\in [0,T]$,
$\p_t$ is positively homogeneous and subadditive (as in (D2) and (D3)), and is (dynamically) monotone and translation invariant
in the following sense:
\begin{itemize}
\item[] Translation Invariance: For all $X\in L^2(\F_T)$ and $m\in
L^\infty(\F_t)$ we have $\p_t(X+m)=\p_t(X)-m$.
\item[] Monotonicity: If $X,Y\in L^2(\F_T)$ and $X\leq Y$ then
$\p_t(X)\geq \p_t(Y).$
\end{itemize}
For a discussion of these axioms see Artzner {\em et al.} (1999). 
Note that by (D1)--(D2) $D_t(m)=0 $ for any $m\in L^2(\F_t)$, so that dynamic 
deviation measures do not satisfy the axiom of monotonicity.
While for dynamic deviation measures time-consistency is defined in terms of 
the generalised conditional variance formula~\eqref{eq:gendev}, in the theory of 
dynamic coherent and convex risk measures 
a recursive tower-type property is the relation strongly time-consistent dynamic risk-measures should 
satisfy. Specifically, a dynamic coherent or convex risk measures is called {strongly time-consistent}, we recall, if
\be
\label{tc}
\p_s(\p_t(X))=\p_s(X) \quad \mbox{ for }s\leq t,
\ee
see for instance among many others Chen and Epstein (2002), Riedel (2004), Delbaen (2006), Artzner {\em et al.} (2007), F\"ollmer and Schied (2011), Cheridito and Kupper (2011). 
Note that a dynamic deviation measure $D$ is not strongly time-consistent (in view of the fact that $D_t(D_T(X))=D_t(0)=0$ for $t<T$). Interestingly, as shown in Proposition~\ref{theoremdual2}, a collection of conditional deviation measures satisfies (D6) if and only if in their dual representations the dual sets are convex, closed, and \emph{additively} $m$-stable, which is a result naturally complementing the well-known fact in the 
literature that the property of time-consistency for coherent risk measures (defined by (\ref{tc})) 
may be characterised in terms of convex, closed, {\em multiplicatively} $m$-stable sets (see Delbaen (2006)).
}
\end{remark}

\noindent{\bf Contents.} The remainder of the paper is organised as follows.
We present in Section~\ref{sec4} the definition of $g$-deviation measures, its properties and a number of examples. 
With these results in hand, we turn in Section~\ref{sec5} to the characterisation of dynamic deviation measures under a domination condition (Theorem~\ref{Main}). 
and proceed to establish in Section~\ref{subsec42}  
an integral representation for general dynamic deviation measures, removing the aforementioned domination condition, (Theorem~\ref{Nondom}) 
and a dual robust representation result. (Theorems~\ref{mstable} and \ref{theoremdual1}). 
In Section~\ref{secap} we phrase a dynamic mean-deviation portfolio-optimisation problem and present an equilibrium solution. It is of interest to investigate other (financial) optimisation problems in terms of dynamic deviation measures, such as optimal hedging problems, capital allocation problems and optimal stopping problems; in the interest of brevity, we leave these as topics for future research.
\smallskip

\section{$g$-deviation measures}
\label{sec4}
In the sequel we assume that the probability space $(\Omega,\F,\mathbb P)$ is equipped with
(i) a standard $d$-dimensional Brownian motion
$W=(W^1,\ldots,W^d)^\intercal$ and (ii) a Poisson
random measure $N(\td t\times \td x)$ on
$[0,T]\times\mathbb{R}^k\setminus\{0\},$ independent of $W$,
with intensity measure $\hat{N}(\td t\times \td x)=\nu(\td x)\td t$, 
where the L\'{e}vy  measure $\nu(\td x)$ satisfies the integrability condition
$$ \int_{\mathbb{R}^k\setminus\{0\}} (|x|^2\wedge 1) \nu(\td x)<\infty,$$
and let $\tilde{N}(\td t \times \td x):=N(\td t \times \td x)-\hat{N}(\td t \times \td x)$ denote the compensated Poisson random measure.
Further, let $\mathcal{U}$ denote the Borel sigma-algebra induced by the $L^2(\nu(\td x))$-norm, $(\F_{t})_{t\in[0,T]}$ the right-continuous completion of the filtration
generated by $W$ and $N$, and $\mathcal{P}$ and $\mathcal O$ the
predictable and optional sigma-algebras on $[0,T]\times\Omega$ with respect to
$(\F_{t})$.  We denote by $ L_d^2(\mathcal{P},\td\mathbb P\times \td t)$ 
the space of all predictable $d$-dimensional processes that are square-integrable with respect to the measure 
$\td\mathbb P\times
\td t$ and we let $\mathcal S^2 = \left\{Y\in\mathcal O: \E{\sup_{0\leq t\leq T}|Y_s|^2} < \infty\right\}$
denote the collection of square-integrable c\`{a}dl\`{a}g optional processes. 
Further, let $\mathcal{B}(\R^k\setminus\{0\})$ be the Borel sigma-algebra on $\R^k\setminus\{0\}.$
For any $X\in L^2(\F_T)$ we denote by $(H^X,\tilde H^X)$ the unique pair of predictable processes with 
$H^X\in L_d^2(\mathcal{P},\td\mathbb P\times \td t)$ and $\tilde{H}^X
\in L^2(\mathcal{P}\times \mathcal{B}(\R^k\setminus\{0\} ),\td\mathbb P\times \td t \times \nu(\td x))$, subsequently 
referred to as the \emph{representing pair} of $X$, satisfying\footnote{See {\em e.g.}\, Theorem III.4.34 in Jacod and Shiryaev (2013)} 
\begin{equation}\label{mrep}
X=\E{X}+\int_0^T H^X_s
\td W_s+\int_0^T\int_{\R^k\setminus\{0\}}\tilde{H}_{s}^X(x)\tilde{N}(\td t \times \td x) ,
\end{equation}
where $\int_0^T H^X_s \td W_s:=\sum_{i=1}^d \int_0^T H^{X,i}_s \td W^i_s$. 
\smallskip

\noindent We consider the following class of driver functions:
\begin{definition}
We call a $\mathcal{P}\otimes
\mathcal{B}(\mathbb{R}^d)\otimes \mathcal{U}$-measurable function $$
\begin{array}{rlclclclll}
g:&[0,T] &\times &\Omega &\times &\mathbb{R}^{d} &\times & L^{2}(\nu(\td x)) &\rightarrow &\R_+\\
&(t, & &\omega , & &h, & & \tilde{h}) &\longmapsto &g(t,\omega ,h,\tilde{h})%
\end{array}
$$
a {\em driver function} if for $\td\mathbb P\times \td t$ a.e. $(\omega,t)\in\Omega\times[0,T]$:
\smallskip
\smallskip

\noindent{\bf (i)} {\em (Positivity)} For any $(h,\tilde h)\in\mathbb R^d\times L^2(\nu(\td x))$ $g(t,h,\tilde{h})\geq 0$ with equality if and only if $(h,\tilde{h})=0$.
\smallskip
\smallskip

\noindent{\bf (ii)} {\em (Lower semi-continuity)} If $h^n\to h$, $\tilde{h}^n\to \tilde h $ $L^2(\nu(\td x))$-a.e. then
$g(t,h,\tilde{h}) \leq \liminf_n g(t,h^n,\tilde{h}^n).$
\end{definition}
\begin{definition}
We call a driver function $g$ {\em convex} if $g(t,h,\tilde{h})$ is convex in $(h,\tilde{h})$, $\td\mathbb P \times \td t$ a.e.; 
{\em positively homogeneous} if $g(t,h,\tilde{h})$ is positively homogeneous in $(h,\tilde{h})$, i.e., for $\lambda> 0$, $g(t,\lambda h,\lambda \tilde{h})=\lambda g(t,h,\tilde{h}),$ $\td\mathbb P\times \td t$ a.e. 
and of {\em linear growth} if for some $K>0$ 
we have $\td\mathbb P\times \td t$ a.e.
\begin{equation}\label{eq:ling}
|g(t,h,\tilde h)|^2 \leq 1 + K^2|h|^2 + K^2\int_{\mathbb R^k\backslash\{0\}} 
\tilde h(x)^2\nu(\td x).
\end{equation}
\end{definition}
To such a driver function $g$ one may associate a corresponding 
 dynamic deviation measure given in terms of the solution to a certain BSDE. 
  \begin{definition}
Let $g$ be a convex and positively homogeneous driver function of linear growth. 
The {\em $g$-deviation measure} 
$D^g = (D^g_t)_{t\in[0,T]}$ is equal to the collection
$D_t: L^2(\mathcal F_T)\to L_+^2(\mathcal F_t)$, $t\in[0,T]$, given by
$$D^g_t(X) = Y_t, \quad\quad X\in L^2(\mathcal F_T),$$ 
where $(Y,Z,\tilde Z)\in \mathcal S^2\times L_d^2(\mathcal{P},\td\mathbb P\times \td t) \times 
L^2(\mathcal{P}\times \mathcal{B}(\R^k\setminus\{0\} ),\td\mathbb P\times \td t \times \nu(\td x))$ 
is the unique solution of the BSDE given in terms of the representing pair $(H^X,\tilde H^X)$ of $X$ by 
\begin{eqnarray}
\label{bsde}
\td Y_t &=& - g(t,H^X_t,\tilde{H}^X_{t})\td t+Z_t \td W_t+\int_{\mathbb R^k\backslash\{0\}} 
\tilde{Z}_t(x)\tilde{N}(\td t \times \td x),\quad t\in[0,T),\\
Y_T &=& 0,
\label{bsde2}
\end{eqnarray} 
\end{definition}
Any $g$-deviation measure admits an integral representation in terms of $g$.  
\begin{Prop}\label{prop:charac} Let $g$ be a convex and positively homogeneous driver function of linear growth.

\noindent{\bf (i)} For given $X\in L^2(\mathcal F_T)$, we have
\begin{equation} \label{charac}
D^g_t(X) =\E{\int_t^Tg(s,H_s^X,\tilde{H}^X_s)\td s\bigg|\F_t},\qquad t\in[0,T].
\end{equation}

\noindent{\bf (ii)} $D^g$ is a  dynamic deviation measure. In particular, $D^g$ satisfies (D6).
\end{Prop}
\begin{proof}{}\ \noindent{\bf (i)}
Letting $Y_t$ be equal to the right-hand side of \eqref{charac} we note that $Y_T = 0$, while we have
\begin{align*}
Y_t = M_t -\int_0^t g(s,H_s^X,\tilde{H}^X_s)\td s, \q M_t = \E{\int_0^Tg(s,H_s^X,\tilde{H}^X_s)\td s\bigg|\F_t}.
\end{align*}
Letting $(Z, \tilde{Z}) = (Z^{M_T}, \tilde{Z}^{M_T})$ the representing pair of $M_T$ 
we have that  $Y_t$ satisfies (\ref{bsde}).

\noindent{\bf (ii)} To verify that (D6) holds we note that the representation \eqref{charac} implies that, for any 
$s,t\in[0,T]$ with $s\leq t$, 
	$$D^g_s(\E{X|\F_t}) = \E{\left.\int_s^t g(u, H^X_u, \tilde{H}^X_u)\td u\right|\F_s},$$ 
	which yields that $D^g_s(\E{X|\F_t}) + \E{D^g_t(X)|\F_s}$ is equal to 
	\begin{eqnarray*}
\E{\left.\int_s^t g(u, H^X_u,
		\tilde{H}^X_u)\td u\right|\F_s} + \E{\left.\E{\left.\int_t^T g(u, H^X_u, \tilde{H}^X_u)\td u\right|\F_t}\right|\F_s}
= \E{\left.\int_s^T g(u, H^X_u, \tilde{H}^X_u)\td u\right|\F_s},
	\end{eqnarray*}
which is equal to $D^g_{s}(X)$.
We show next that the axioms (D1)--(D5) are satisfied. We note from \eqref{charac} that $D_t^g(X+m)=D_t^g(X)$
for any $X\in L^2(\F_T)$, $m\in L^\infty_+(\F_t)$ while $D_t^g(m)=0$ as $g(t,0,0) = 0$, so that
(D1) holds. Using \eqref{charac} we see that $D^g$ inherits
the properties of convexity and positive homogeneity from $g$, so that (D2) and (D3) are satisfied.
Positivity (D4) is straightforward to verify by using that $g$ is nonnegative and strictly positive for $(h,\tilde{h})\ne 0$. 
Finally, noting that (a) if $X^n\to X$ in $L^2(\F_T)$, $(H^{X^n},\tilde{H}^{X^n})$ converges to $(H^{X},\tilde{H}^{X})$ in $L_d^2(\td\mathbb P\times \td t)\times L^2(\td\mathbb P\times \td t \times \nu(\td x))$-norm and that (b) $g$ is nonnegative and lower semi-continuous, we have 
by an application of Fatou's Lemma 
\begin{align*}
\liminf_nD^g_t(X^n)& =   \liminf_n \E{\left.\int_t^T g(s,H^{X^n}_s,\tilde{H}^{X^n}_s)\td s\right|\F_t} \geq
\E{\left.\int_t^T\liminf_n  g(s,H^{X^n}_s,\tilde{H}^{X^n}_s)\td s\right|\F_t}\\ &\geq 
\E{\left.\int_t^T g(s,H^{X}_s,\tilde{H}^{X}_s)\td s\right|\F_t} =D^g_t(X),
\end{align*}
which shows that also the lower-semicontinuity condition in
(D5) is satisfied.
\end{proof}

\noindent The linear growth condition and convexity guarantee that a $g$-deviation measure is continuous in $L^2$.
\begin{lemma}
\label{lemmacont}
Let $g$ be a convex driver function of linear growth.
If $X^n$ converge to $X$ in $L^2(\F_T)$ then $\lim_n {D}^g_0(X^n)={D}^g_0(X).$
\end{lemma}
\begin{proof}{}
If $X^n$ converge to $X$ in $L^2(\F_T)$ then, as noted before, $H^{X^n}$ and $\tilde{H}^{X^n}$
converge to $H^X$ and $\tilde{H}^X$ in $L^2_d(\td\mathbb P\times \td t)$ and $L^2(\td\mathbb P\times
\td t\times \nu(\td x))$ norms.
Next note that $|g(s, H^{X^n}_s, \tilde{H}^{X^n}_s)|$
is a uniformly integrable sequence by the growth-condition \eqref{eq:ling} 
and the convergence of the processes $|H^{X^n}|^2$ and $\IR|\tilde{H}^{X^n}|^2(x)
\nu(\td x)$ in $L^1(\td\mathbb P\times \td t)$-norm. As $g$ is continuous (as it is convex and locally bounded, {\em cf.} Theorem 2.2.9 in 
Zalinescu (2002))  it follows thus that
$\lim_n {D}_0^g(X^n)=
\lim_n \E{\int_0^T g(s, H^{X^n}_s, \tilde{H}^{X^n}_s)\td s}=\E{\int_0^T g(s,
H^{X}_s, \tilde{H}^{X}_s)\td s}=D^g_0(X)$.
\end{proof}

\noindent We list a number of properties of a $g$-deviation measure that are characterised 
in terms of those of the driver function $g$.

\begin{Prop}
\label{theoremequivalent}
Let $g$ and $\tilde g$ be driver functions of linear growth.

\smallskip\smallskip

\noindent{\bf (i)} $D^g$ is conditionally convex if and only if $g$ is convex.
\smallskip\smallskip

\noindent{\bf (ii)} $D^g$ satisfies (D2) if and only if $g$ is positively
homogeneous.
\smallskip\smallskip

\noindent{\bf (iii)} $D^g$ is symmetric, that is, $D^g_t(X) = D^g_t(-X)$ for all $t$, if and only if
$g$ is symmetric in $(h,\tilde{h}).$
\smallskip\smallskip

\noindent{\bf (iv)} $D^g \ge D^{\tilde g}$ if and only if $g\ge \tilde g$ $\td\mathbb P\times \td t$ a.e. 
\end{Prop}
To simplify notation we denote, for $s,t\in[0,T]$ with $s\leq t$ 
and $(H,\tilde{H})\in L_d^2(\mathcal{P},\td\mathbb P\times \td t \times 
L^2(\mathcal{P}\times \mathcal{B}(\R^k\setminus\{0\} ),\td\mathbb P\times \td t \times \nu(\td x))$, 
 $(H\cdot W)_{s,t}:= \int_s^t H_u \td W_u$
and  $(\tilde{H}\cdot\tilde{N})_{s,t}:=
\int_{(s,t]\times\R^k\setminus\{0\}}\tilde{H}_{u}(x)\tilde{N}(\td u\times\td x),$
and moverover $(H\cdot W)_{t}:= (H\cdot W)_{0,t}$ 
and $(\tilde{H}\cdot\tilde{N})_{t}:=
(\tilde{H}\cdot\tilde{N})_{0,t}.$

\noindent{\it Proof of Proposition \ref{theoremequivalent}.} 
First, we prove (i)`$\Rightarrow$' by contradiction.
Suppose that there exist predictable processes $B^i$ and $\tilde{B}^i$ for
$i=1,2$, a nonzero predictable set $C$ and a $\lambda\in (0,1)$ such that 
for $(s,\omega)\in C$
$$g(s,\lambda B^1_s + (1-\lambda) B^2_s, \lambda \tilde{B}_s^1 + (1-\lambda) \tilde{B}_s^2) > \lambda g(s,B^1_s,\tilde{B}^1_s) + (1-\lambda)
g(s,B^2_s,\tilde{B}^2_s).$$ 
Set $H^i_s(\omega)=B^i_s(\omega)$, $i=1,2$, if $(s,\omega) \in C$ and zero otherwise,
 define $\tilde{H}^i$, $i=1,2$, similarly and set $X = (H^1
\cdot W)_T + (\tilde{H}^1 \cdot \tilde{N})_T$, $Y = (H^2 \cdot W)_T +
(\tilde{H}^2 \cdot \tilde{N})_T$ and $C_s=\{\omega\in\Omega: (s,\omega)\in C\}$. 
using that $g(s,0,0)=0$ it follows that $D^g_0(\lambda X + (1-\lambda) Y)$ 
is equal to
\begin{eqnarray}\nonumber
\lefteqn{\E{\int_0^Tg(s,\lambda I_{C_s} H_s^1 +
(1-\lambda) I_{C_s}H_s^2, \lambda I_{C_s} \tilde{H}_s^1 + (1-\lambda) I_{C_s}
\tilde{H}_s^2)\td s}}\\ \nonumber
&=& \E{\int_0^TI_{C_s} g(s,\lambda H_s^1 + (1-\lambda) H_s^2, \lambda
\tilde{H}_s^1 + (1-\lambda) \tilde{H}_s^2)\td s}\\ \nonumber
&>& \lambda\E{\int_0^T I_{C_s} g(s,H^1_s,\tilde{H}^1_s)\td s} +
(1-\lambda)\E{\int_0^T I_{C_s} g(s,H^2_s,\tilde{H}^2_s)\td s} 
\\ \label{lEil}
&=&\lambda\E{\int_0^T g(s,I_{C_s}H^1_s,I_{C_s} \tilde{H}^1_s)\td s} +
(1-\lambda)\E{\int_0^T g(s,I_{C_s}H^2_s,I_{C_s}\tilde{H}^2_s)\td s}.
\end{eqnarray}
The right-hand side of \eqref{lEil} is equal to $\lambda D^g_0(X) + (1-\lambda)D^g_0(Y)$, 
in contradiction to the convexity of $D^g_0$.
The directions `$\Rightarrow$' in (ii), (iii) and (iv) follow by similar lines of reasoning. 
The implications `$\Leftarrow$'  in (i)--(iv) follow from~\eqref{charac} in Proposition~\ref{prop:charac}.
\qed
\bigskip

\noindent{\bf Examples.} We give next a number of examples of $g$-deviation measures.

\begin{Example}\label{Dl} The family of $g$-deviation measures  
with driver functions given by
	\begin{equation}\label{gl}
g_{c,d}(t,h,\tilde h) = c\, |h| + d\, \sqrt{\IR |\tilde{h}(x)|^2
		\nu(\td x)}, \q c,d\in\mathbb R_+\backslash\{0\}, 
		\end{equation}
corresponds to a measurement of the risk of a random variable $X\in L^2(\F_T)$ by the integrated multiples 
of the local volatilities of the continuous and discontinuous martingale parts in 
its martingale representation~\eqref{mrep}. 
\end{Example}
\begin{Example}\label{D2}
In the case of a $g$-deviation measure with driver function given by 
$$
g(\omega, t,h,\tilde{h}) = {\it CVaR}^{\nu}_{t,a}(\tilde h), \qquad a\in(0,\nu(\mathbb R^k\backslash\{0\})),
$$
the risk is measured in terms of the values of the (large) jump sizes under $CVaR^\nu_{t,a}$.
Here $CVaR^{\nu}_{t,a}(\tilde h) = \frac{1}{a}\int_0^a 
VaR^{\nu}_{t,b}(\tilde h)\td b
$
is given in terms of the left-quantiles $VaR^{\nu}_{t,a}(\tilde h)$, $a\in(0,\nu(\mathbb R^k\backslash\{0\}))$ 
of $h(J)$ under the measure $\nu(\td x)$, that is,
$$
VaR^{\nu}_{t,a}(\tilde h) := VaR^{\nu}_{a}(h(J)) := \sup\{y\in\mathbb R: 
\nu(\{x\in \mathbb R^k\backslash\{0\}: \tilde h(x) < -y\}) < a \}.
$$
\end{Example}
In the next example we deploy the following auxiliary result:
\begin{proposition}\label{prop:disc}\label{Lemmaadditive}\label{prop}  
Let $I:=\{t_0, t_1, \ldots, t_n\}\subset[0,T]$ be strictly ordered. 
$D = (D_t)_{t\in I}$ satisfies (D1)--(D4) and (D6) 
if and only if for some collection $\tilde D = (\tilde D_t)_{t\in I}$ of conditional deviation measures 
we have 
\begin{equation}\label{DDH}
D_t(X) = \E{\left. \sum_{t_i\in I: t_i\ge t} 
\tilde D_{t_i}\left(\E{X|\mathcal F_{t_{i+1}}} - \E{X|\mathcal F_{t_{i}}}\right)\right|\mathcal F_t}, \q t\in I,\ 
X\in L^2(\mathcal F_T).
\end{equation}
In particular, a dynamic deviation measure $D$ satisfies \eqref{DDH} with $\tilde D_{t_i} = D_{t_i}$, $t_i\in I$.
\end{proposition}
\begin{proof}{}%
\ `$\Leftarrow$': We will only show that $D_t$ satisfies (D6), as 
it is clear that (D1)--(D4) are satisfied. Let $X\in L^2(\mathcal F_T)$ and note that 
as $\tilde D_t$, $t\in I$, satisfy (D1) and (D4) we have for any $s,t\in I$ with $s>t$ 
that $D_t(\E{X|\mathcal F_s}) = \sum_{t_i\in I:t\leq t_i< s}\E{\tilde D_{t_i}(\E{X|\mathcal F_{t_{i+1}}})|\F_t}$.
Thus, we have that $D_t(X)$ is equal to
$$
\sum_{t_i\in I:t\leq t_i< s}\E{\left.\tilde D_{t_i}(\E{X|\mathcal F_{t_{i+1}}})\right|\F_t} 
+ \sum_{t_i\in I:s\leq t_i}\E{\left.\tilde D_{t_i}(\E{X|\mathcal F_{t_{i+1}}})\right|\F_t}
= D_t(\E{X|\mathcal F_s}) + \E{D_s{(X)}|\mathcal F_t}.
$$
`$\Rightarrow$': 
For $X\in L^2(\mathcal F_T)$ and $t_{i-1}\in I$, $i\ge 1$, 
we have by (D6) and (D1)
\begin{align}\nonumber
D_{t_{i-1}}(X) &= D_{t_{i-1}}(\E{X|\F_{t_{i}}}) + \E{D_{t_{i}}(X)|\F_{t_{i-1}}}\\
&= D_{t_{i-1}}(\E{X|\F_{t_{i}}} - \E{X|\mathcal F_{t_{i-1}}} ) + \E{D_{t_{i}}(X)|\F_{t_{i-1}}}.
\label{pdp}
\end{align}
An induction argument based on \eqref{pdp} then yields 
that \eqref{DDH} holds with $\tilde D_t=D_t$, $t\in I$.
\end{proof}
\begin{Example}\label{D3} The formula~\eqref{DDH} in Proposition~\ref{prop:disc} gives a way to define a collection 
$D=(D_t)_{t\in I}$ satisfying axioms (D1)--(D6) for $s,t\in I$, which we call a {\em dynamic deviation 
measure on the grid $I$}. Comparison of \eqref{DDH} and \eqref{charac} suggests that one may obtain 
the values of a dynamic deviation measure as limit of the values of (suitably chosen) dynamic deviation 
measures on grids with vanishing mesh sizes. 
We next illustrate this for the $g$-deviation measures $\bar D^{\lambda}:=D^{g_\lambda}$, $\lambda>0$, 
corresponding to the driver functions $g_\lambda$ given by 
\begin{equation}\label{gcd}
g_\lambda(\omega, t,h,\tilde{h}) := \lambda \sqrt{|h|^2+\IR |\tilde{h}(x)|^2\nu(\td x)},\quad\quad \lambda>0,
\end{equation}
and random variables $X\in L^2(\F_T)$ of the form 
\begin{equation}\label{Xfg}
X = x + \int_0^T f(t)\td W_t + \int_{[0,T]\times\mathbb R^k\backslash\{0\}}g(t,y) \tilde N(\td t\times\td y)
\end{equation}
with $x\in\mathbb R$, $f\in C([0,T],\mathbb R^d)$ and $g\in C_0([0,T]\times\mathbb R^k, \mathbb R)$\footnote{$C([0,T],\mathbb R^d)$ and $C_0([0,T]\times\mathbb R^k,\mathbb R)$ denote the sets of continuous functions $f:[0,T]\mapsto\mathbb R^d$, and  of continuous functions $g:[0,T]\times\mathbb R^k\mapsto\mathbb R$ that are such that $\sup_{t\in[0,T]}|g(t,x)|\to 0$ as $|x|\to\infty$ and $\sup_{x\in\mathbb R^k\backslash\{0\}}\sup_{t\in[0,T]}\{|g(t,x)|/|x|^2\} <\infty$.}. We 
construct approximating sequences in terms of 
the conditional $CVaR$-deviation measures  given  by
$\tilde D_t(Y) := CVaR_{t,\alpha}(Y - \E{Y|\F_t})$ for $Y\in L^2(\F_T)$, $t\in[0,T]$, $\alpha\in(0,1)$, 
where for $Z\in L^2(\F_T)$  
$$CVaR_{t,\alpha}(Z)=\frac{1}{\alpha}\int_0^\alpha VaR_{t,b}(Z)db,\q 
VaR_{t,b}(Z) = \sup\{y\in\mathbb R: \mathbb P(Z<-y|\F_t)<b\},
$$  see Rockafellar \textit{et al.} (2006a).\\
Specifically, the expression in \eqref{DDH} suggests to \emph{scale} the value of 
conditional deviation measures corresponding to small time units in order to obtain in the limit 
a dynamic deviation measure. Denoting for $X$ of the form~\eqref{Xfg} 
$$M_{t_{i+1}}:=\E{X|\F_{t_{i+1}}},\q \Delta M_{i+1}:=M_{t_{i+1}}-M_{t_{i}},\q t_i=Ti/2^n,\q i=0,\ldots,2^n-1,$$ 
with $t_{2^n}=T$ and following this suggestion we specify the contribution to the total 
risk of 
$$\Delta M_{i+1} = \int_{t_{i}}^{t_{i+1}} f(s)\td W_s
 + \int_{(t_i,t_{i+1}]\times(\mathbb R^k\backslash\{0\})}g(s,y)\tilde N(\td s\times\td y),\q i=0,\ldots,2^n-1, 
$$ 
 by
$
\tilde D_{t_i}(\Delta M_{i+1}) := \sqrt{\Delta t_{i+1}}{CVaR}_{t_{i},\alpha}(\Delta M_{i+1}),$ 
$\Delta t_{i+1} = t_{i+1}- t_i$, 
which gives rise to the dynamic deviation measure $D^{(n)}=(D^{(n)}_t)_{t\in I_n}$ 
on $I_n:=\{t_i, i=0,\ldots, 2^n\}$ given by
\begin{eqnarray}\nonumber
D^{(n)}_t(X) &=& \sum_{t_i\ge t}\E{\left.\tilde D_{t_i}(\Delta M_{i+1})\right|\F_{t}} = 
 	\sum_{t_i\geq t}\sqrt{\sigma^2(t_i)}\Delta t_{i+1}\E{\left.CVaR_{t_{i},\alpha}\left(\frac{\Delta M_{i+1}}{\sqrt{\sigma^2(t_i)\Delta t_{i+1}}}\right)\right|\F_{t}},\\
&& \text{with}\ \sigma^2(t) := |f(t)|^2 + \int_{\mathbb R^k\backslash\{0\}}|g(t,x)|^2\nu(\td x),\ t\in I_n,
\label{Dns}
\end{eqnarray} 
where we used that $CVaR_{t_{i},\alpha}$ is positively homogeneous. As $\Delta M_{i+1}$ is infinitely divisible 
and $f$ and $g$ are bounded,  we have by an application of 
Lindeberg-Feller Central Limit Theorem (see {\em e.g.,} Durrett (2004), p.129) 
that, when we let $n\to\infty$ while keeping $t_i$ fixed the ratio $\Delta M_{i+1}/\sqrt{\sigma^2(t_i)\Delta t_{i+1}}$ 
converges in distribution to a standard normal random variable $\xi$.  
By uniform integrability and the independence of $\Delta M_{i+1}$ from $\F_{t_i}$
we have 
$$CVaR_{\alpha,t_i}\left(\frac{\Delta M_{i+1}}{\sqrt{\sigma^2(t_i)\Delta t_{i+1}}}\right) = 
CVaR_{\alpha}\left(\frac{\Delta M_{i+1}}{\sqrt{\sigma^2(t_i)\Delta t_{i+1}}}\right) 
\to CVaR_{\alpha}(\xi) = \frac{1}{\alpha}\int_0^\alpha \Phi^{-1}(u)\td u =: c_\alpha,$$ 
where $CVaR_{\alpha}(\cdot) = CVaR_{\alpha,0}(\cdot)$ and  
$\Phi^{-1}$ denotes the inverse of the standard normal distribution function $\Phi$. 	
Hence, letting $n\to\infty$ in \eqref{Dns} and deploying the uniform continuity of $f$ and $g$ 
we have for any $t\in[0,T]$ of the form $t=k/2^m$, $k,m\in\mathbb N$, 
	\begin{equation}\label{DDRM}
	D^{(n)}_t(X) \to c_\alpha\, \mathbb E\left[\left. \int_t^T \sqrt{|f(s)|^2 + \int_{\mathbb R^k\backslash\{0\}}|g(t,x)|^2\nu(\td x)} \td s \right|\mathcal F_t\right] = 
	\bar D^{{c_\alpha}}_t(X).
	\end{equation}
	\end{Example}

\section{Characterisation theorem}\label{sec5}

We show next that any  dynamic
deviation measure that satisfies a domination condition is a 
$g$-deviation measure for some driver function $g$.

\begin{definition}
A  dynamic deviation measure $D = (D_t)_{t\in[0,T]}$ 
is called {\em $\lambda$-dominated} if for all $t\in[0,T]$ and $X\in L^2(\mathcal F_T)$ we have
$$D_t(X) \leq \bar D^{\lambda}_t(X).$$
\end{definition}

\begin{theorem}
\label{Main} Let $D=(D_t)_{t\in[0,T]}$ be a collection of maps $D_t:L^2(\mathcal F_T)\to L_+^2(\mathcal F_t)$, $t\in[0,T]$.
Then $D$ is a dynamic deviation measure that is $\lambda$-dominated for some $\lambda> 0$
if and only if there exists a convex and positively homogeneous driver function $g$ of linear growth 
such that $D=D^g$. Furthermore, this driver function $g$ is unique $\td\mathbb P\times \td t$ a.e.
\end{theorem}
\begin{proof}{}
We first verify uniqueness: If $\bar{g}$ is a driver function that satisfies $D^g = D^{\bar g}$, 
it follows from Proposition~\ref{theoremequivalent}(iv) that $g=\bar g$ $\td\mathbb P\times \td t$ a.e.
We note next that the implication `$\Leftarrow$' follows from Proposition~\ref{prop:charac}.
The remainder is devoted to the proof of the implication `$\Rightarrow$', 
which is established using a number of auxiliary results (the proofs of which are deferred to the end of the  section). 

Thus, let $D$ be a given  dynamic deviation measure that is $\lambda$-dominated, 
so that in particular $D_0$ is finite.  
We identify next a candidate driver function $g$. For the remainder of the proof we assume for the ease of presentation 
that $d=1$.
For fixed $h\in \mathbb{R}$ and
$\tilde{h}\in L^2(\nu(\td x))$ consider the mapping
$\mu_{h,\tilde{h}}: \mathcal{P}\times \mathcal{P}\to \R$ given by
$$\mu_{h,\tilde{h}}:C_1\times C_2\mapsto D_0
\left( (I_{C_1} h\cdot W)_T + (I_{C_2}\tilde h\cdot \tilde N)_T\right).$$

\begin{lemma}\label{lem:muC} Let $(h,\tilde h) \in\mathbb R\times L^2(\nu(\td x))$.\ 

\noindent{\bf (i)} $C \mapsto \mu_{h,\tilde{h}}(C,\emptyset),
C \mapsto \mu_{h,\tilde{h}}(\emptyset,C)$ and $C \mapsto \mu_{h,\tilde{h}}(C,C)$ are
$\s$-finite measures on $([0,T]\times\Omega,\mathcal{P})$.

\noindent{\bf (ii)} For any $C_1,C_2 \in \mathcal{P}$ we have
\be
\label{decompose}
\mu_{h,\tilde{h}}(C_1,C_2) = \mu_{h,\tilde{h}}(C_1 \setminus C_2,\emptyset) + \mu_{h,\tilde{h}}(\emptyset,C_2 \setminus
C_1) + \mu_{h,\tilde{h}}(C_1\cap C_2,C_1\cap C_2).\ee
\end{lemma}
As $D_0$ is $\lambda$-dominated  $C\mapsto \mu_{h,\tilde{h}}(C,C)$ is absolutely continuous with respect to
the measure $\td\mathbb P \times \td t$ and we conclude from the Radon-Nikodym theorem
that there exist an integrable non-negative density, say $R_{h,\tilde{h}}(s,\omega)$, 
that is such that $R_{0,0}=0$ and for any set $C\in \mathcal{P}$ 
\begin{eqnarray}
\label{Corroproof}
\mu_{h,\tilde{h}} (C, C) &=&\E{\int_0^T I_{C_s} R_{h,\tilde{h}}(s)
\td s},
\end{eqnarray}
where $C_s=\{\omega\in\Omega: (\omega,s) \in C \}$. 
In particular, we note that $\mu_{h,\tilde{h}}(C,\emptyset)=\mu_{h,0}(C,C)$ and
$\mu_{h,\tilde{h}}(\emptyset,C)=\mu_{0,\tilde{h}}(C,C)$ 
 satisfy \eqref{Corroproof}
with $R_{h,\tilde{h}}$ replaced by $R_{0,\tilde h}$ 
and $R_{h,0}$ respectively. We define the candidate driver function $g$ in terms of $R$ by
\be
\label{g}
g(t,\omega,h,\tilde{h}) := R_{h,\tilde{h}}(t,\omega), 
\qquad (t,\omega)\in[0,T]\times\Omega.\ee 
 The next result confirms that $g$ is a driver function.

\begin{lemma}
\label{lemmaconvex2}
There exists a version of $g$ such that, for $\td\mathbb P\times \td t$ a.e.~$(t,\omega)$, $(h,\tilde{h})\mapsto g(t,\omega,h,\tilde h)$ is continuous, convex, positively-homogeneous and dominated by $g_\lambda$. 
\end{lemma}

Note that $(t,\omega)\mapsto
g(t,\omega,h,\tilde{h})$ is predictable for every $(h,\tilde{h})\in\mathbb R\times L^2(\nu(\td x))$ and by Lemma~\ref{lemmaconvex2} $(h,\tilde{h})\mapsto g(t,\omega,h,\tilde h)$ is continuous in $(h,\tilde{h})$,   
so that by standard arguments $g$ can be approximated by $\mathcal{P} \otimes
\mathcal{B}(\mathbb{R^d}) \otimes \mathcal{U}$-measurable step functions and
$g$ itself may seen to be $\mathcal{P} \otimes \mathcal{B}(\mathbb{R^d})
\otimes \mathcal{U}$-measurable. Note further that $g(t,\omega,h,\tilde{h})$ is non-negative as
$R_{h,\tilde{h}}(t,\omega)$ is so for each $(h,\tilde{h})$, and
$g(s,\omega,0,0)=0$ since the density  $R_{0,0}(s,\omega)$ of the measure
$\mu_{0,0}$ is zero.
\noindent In the next result we show that $D_0$ may be identified with $D^g_0$.

\begin{lemma}
\label{lemmaeq}
Let $g$ be as in Lemma~\ref{lemmaconvex2}. 
For $X\in L^2(\mathcal F_T)$ we have $D_0(X) = D_0^g(X)$.
\end{lemma}

\noindent Lemma~\ref{lemmaeq} and Remark~\ref{rem1}(ii) imply that $D_t= D^g_t$ not only for $t=0$ but also for all other
$t\in(0,T]$. The proof is complete.
\end{proof}
\subsubsection*{Proofs of Lemmas~\ref{lem:muC}, \ref{lemmaconvex2} and \ref{lemmaeq}}
The proof of Lemma~\ref{lem:muC} is based on the following auxiliary result:
\begin{proposition}\label{prop:dis}
Let $D$ be a dynamic deviation measure and $t\in[0,T]$. 
If $A_1,\ldots,A_n\in \F_t$ and
$A_i \cap A_j = \emptyset$ for $i\neq j$ and $X_1, \ldots, X_n \in L^2(\F_T)$, then for any $t\in[0,T]$
\be
\label{local}
D_t\left(\sum_{i=1}^n I_{A_i}X_i\right)=\sum_{i=1}^n D_t( I_{A_i}X_i).
\ee
\end{proposition}
\begin{proof}{} Set $S_k:=\sum_{i=1}^{k}I_{A_i}X_i$ and $B_k=\cup_{i=1}^{k}A_i$, $k=1,\ldots, n$.
Let us first show by an induction argument that
\be
\label{missed}
D_t\left(S_n\right) = \sum_{i=1}^n I_{A_i}D_t(X_i).
\ee
Eqn.~(\ref{local}) is a direct consequence of (\ref{missed}) and (\ref{localproperty}).
Using (\ref{localproperty}) and the fact $B_{n-1}\cap A_n = \emptyset$ we have 
\begin{align*}
D_t\left(S_n\right) & = D_t(I_{B_{n-1}} S_{n-1}+I_{B_{n-1}^c}I_{A_n}X_n) = I_{B_{n-1}} D_t(S_{n-1}) + I_{B_{n-1}^c} D_t(I_{A_n}X_n)\\
&= I_{B_{n-1}}\sum_{i=1}^{n-1} I_{A_i}D_t(X_i) + I_{B_{n-1}^c}I_{A_n}D_t(X_n) =
\sum_{i=1}^{n} I_{A_i}D_t(X_i),
\end{align*}
where we used (\ref{localproperty}) and the induction assumption in the third equality. 
This completes the proof of \eqref{missed} and hence of the Lemma.
\end{proof}

\begin{proof}{\ of Lemma~\ref{lem:muC}}
\noindent{\bf (i)} Let us first show that $C \mapsto \mu_{h,\tilde{h}}(C,\emptyset)$ constitutes a
$\s$-finite measure. Clearly, $\mu_{h,\tilde{h}}(\cdot,\emptyset)$ is non-negative and
$\mu_{h,\tilde{h}}(\emptyset,\emptyset)=0$.
Next we verify that $C\mapsto \mu_{h,\tilde{h}}(C,\emptyset)$ is additive for
disjoint sets of the form $C_1:=(t_1, t_2]\times A$ and $C_2:=(t_3, t_4]\times B$ with
$A\in\F_{t_1}$ and $B\in\F_{t_3}$. We consider first 
the case $t_1 \leq t_3 \leq t_2 \leq t_4$ and $A\cap B=\emptyset$ (note that in this case
$C_1\cap C_2=\emptyset$). 
By deploying Propositions~\ref{Lemmaadditive} and \ref{prop:dis}
 we note that 
$\mu_{h,\tilde{h}}\left(((t_1, t_2]\times A)\cup((t_3,t_4]\times
B),\emptyset\right)$ is equal to
\begin{eqnarray*}
\lefteqn{D_0\left(I_A h\cdot W)_{t_1,t_3} + (I_{A\cup B} h \cdot W)_{t_3,t_2} +
(I_B h\cdot W)_{t_2,t_4}\right)}\\
&=& \E{D_{t_1}(( I_A h \cdot W)_{t_1,t_3})} +
\E{D_{t_3}((I_{A\cup B} h \cdot W)_{t_3,t_2})} +
\E{D_{t_2}((I_B h \cdot W)_{t_2,t_4})}
\\
&=& \E{D_{t_1}\left((I_A h\cdot W)_{t_1,t_3}\right)} 
+ \E{D_{t_3}\left((I_A h\cdot W)_{t_3,t_2}\right)
}\\&& + \E{D_{t_3}\left((I_B h\cdot W)_{t_3,t_2}\right)
} 
+ \E{D_{t_2}\left((I_B h\cdot W)_{t_2,t_4}\right)
}\\
&=& D_0\left((I_A h\cdot W)_{t_1,t_2}\right) + D_0\left((I_B h\cdot W)_{t_3,t_4}\right), 
\end{eqnarray*}
which is equal to 
$\mu_{h,\tilde{h}}((t_1, t_2]\times A,\emptyset) + \mu_{h,\tilde{h}}((t_3, t_4]\times B,\emptyset) $.
The cases $t_1 \leq t_2 < t_3 \leq t_4$ and  
$t_1 \leq t_3 \leq t_4 \leq t_2$ may be verified in a similar manner.
Thus, we may conclude that $\mu_{h,\tilde{h}}$ is additive on disjoint
sets of the form $(t_1, t_2]\times A$ and $(t_3, t_4]\times B$. 
As $D_0$ is
continuous in $L^2(\F_T)$ (see Remark~\ref{rem1}(iii))  and the collection of sets considered above is a semi-algebra
generating the predictable $\s$-algebra it follows that
$\mu_{h,\tilde{h}}(\cdot,\emptyset)$ is $\sigma$-finite.
The proofs that $C\mapsto\mu_{h,\tilde h}(\emptyset,C)$ and $C\mapsto\mu_{h,\tilde h}(C,C)$
are $\sigma$-finite measures are analogous, replacing in the equations 
above the term $h \cdot W$ by $
\tilde{h} \cdot \tilde{N}$ and $(h \cdot W +\tilde{h} \cdot \tilde{N})$, respectively.

\noindent{\bf (ii)} Define $C_1, C_2$ as in (i) and consider the case $t_1 \leq t_3 \leq t_2 \leq t_4$ 
with general (not necessarily disjoint) $A\in\F_{t_1}$ and $B\in\F_{t_3}$.
Expressing $X = I_A(h \cdot W)_{t_1,t_2} + I_B( \tilde{h} \cdot \tilde{N})_{t_3,t_4}$ 
as the sum of martingale increments
$$
X = I_A(h \cdot W)_{t_1,t_3} +I_{A\setminus B }(h \cdot W)_{t_3,t_2}+
I_{A\cap B }[(h \cdot W)_{t_3,t_2}+( \tilde{h} \cdot
\tilde{N})_{t_3,t_2}]+I_{B\setminus A}(\tilde{h}\cdot\tilde{N})_{t_3,t_2}
+ I_B( \tilde{h} \cdot \tilde{N})_{t_2,t_4}
$$
and using  Propositions~\ref{Lemmaadditive} and \ref{prop:dis} 
we have that $\mu_{h,\tilde{h}}(C_1,C_2)=D_0(X)$ is equal to
\begin{multline*}
\E{D_{t_1}(I_A( h \cdot W)_{t_1,t_3})} + {\mathbb E}\bigg[ D_{t_3}\bigg(
I_{A\setminus B}(h \cdot W)_{t_3,t_2}+
I_{A\cap B }[(h \cdot W)_{t_3,t_2}+( \tilde{h} \cdot \tilde{N})_{t_3,t_2}]
+I_{B\setminus A}(\tilde{h}\cdot\tilde{N})_{t_3,t_2}\bigg )
\bigg] \\ + \E{D_{t_2}(I_B( \tilde{h} \cdot \tilde{N})_{t_2,t_4})} \\
= \E{D_{t_1}(I_A(h \cdot W)_{t_1,t_3} )} + \E{ D_{t_3}(I_{A\setminus B}(h
\cdot W)_{t_3,t_2} )}+ \E{ D_{t_3}(I_{A\cap B} [(h\cdot
W)_{t_3,t_2}+(\tilde{h}\cdot \tilde{N})_{t_3,t_2}])}\\
+ \E{ D_{t_3}(I_{B\setminus A}( \tilde{h} \cdot
\tilde{N})_{t_3,t_2})} + \E{D_{t_2}(I_B(\tilde{h} \cdot
\tilde{N})_{t_2,t_4})}.
\end{multline*}
Thus, using Proposition \ref{Lemmaadditive} again we have
\begin{align*}
\mu_{h,\tilde{h}}\left(C_1, C_2\right) &= D_0(I_A( h \cdot W)_{t_1,t_3}+I_{A\setminus B}( h\cdot W)_{t_3,t_2}) 
 + D_0(I_{B\cap A} [(h\cdot W)_{t_3,t_2}+(\tilde{h}\cdot
\tilde{N})_{t_3,t_2}]) \\
&\hspace{0.5cm}+ D_0(I_{B\setminus A} ( \tilde{h}\cdot \tilde{N})_{t_3,t_2}+
I_B(\tilde{h}\cdot \tilde{N})_{t_2,t_4})\\
&=\mu_{h,\tilde{h}}\left(C_1\setminus
C_2,\emptyset\right)+\mu_{h,\tilde{h}}(C_1\cap C_2,C_1\cap
C_2)+\mu_{h,\tilde{h}}\left(\emptyset, C_2\setminus C_1\right).
\end{align*}
The cases $t_1\leq t_2< t_3\leq t_4$
and $t_1\leq t_3\leq t_4\leq t_2$ may be verified in a similar manner.
By the continuity of $D_0$ (Remark~\ref{rem1}(iii)) and
monotone class arguments (by keeping first  $C_1$ and then $C_2$ fixed) it follows that (\ref{decompose})
holds for all predictable sets, as asserted.
\end{proof}

\begin{proof}{\ of Lemma~\ref{lemmaconvex2}}
First of all, note that the predictable $\sigma$-algebra is generated by
countable many sets, say $A_1,A_2,\ldots.$ Fix $n\in\mathbb N$ and denote
$\mathcal{P}^n:=\sigma(A_1,\ldots,A_n).$
By considering finer partitions we may after relabeling assume without loss of
generality that the $A_i$ are disjoint. Denote by $\eta$ the measure $\eta:=\td\mathbb P\times \td t$ 
on $(\Omega\times[0,T],\mathcal P)$ and let
 $R^{n}_{h,\tilde{h}}={\it E}_\eta[R_{h,\tilde{h}}|\mathcal{P}^n]$.\footnote{Specifically, 
$R^{n}_{h,\tilde{h}}$ is the $\mathcal P^n$-measurable random variable 
satisfying $E_\eta[R_{h,\tilde h}U]=E_\eta[R^n_{h,\tilde h} U]$ for all bounded $\mathcal P^n$ 
random variables $U$, with $E_\eta[Z]=\int_0^TE[Z(s)]\td s$ for $Z\in L^1(\eta)$. 
 }
Since the filtration is generated by the disjoint sets $A_1,A_2,\ldots, A_n$ it is
standard to note that
\be
\label{convex}
R^{n}_{h,\tilde{h}}(s,\omega)=\sum_{i:\nu(A_i)\neq 0}
\frac{I_{A_i}(s,\omega)}{\eta(A_i)}\mu_{h,\tilde{h}}(A_i,A_i) \mbox{ for
}\td\mathbb P\times \td s \mbox{ a.e. }(s,\omega) .\ee
By possibly modifying $R^n_{h,\tilde{h}}$ on a zero-set we may assume that (\ref{convex}) holds for all $(s,\omega)\in[0,T]\times\Omega.$ 
It follows from (\ref{convex}) and the convexity and positive homogeneity of 
$(h,\tilde{h})\to \mu_{h,\tilde{h}}(A_i,A_i)$ that, for all fixed
$(s,\omega)$, $R^{n}_{h,\tilde{h}}(s,\omega)$ is convex and positively homogeneous in $(h,\tilde{h})$. 
Furthermore, we claim that $|R^n_{h,\tilde{h}}|\leq g_\lambda(h,\tilde{h})$. 
For suppose this were not the case, that is, for some
 $(h,\tilde{h})$ and $A_i$, $|R^n_{h,\tilde{h}}|> g_\lambda(h,\tilde{h})$ 
	for all $(s,\omega)\in A_i$. Then we would have for $X = (H \cdot W)_T + ( \tilde{H} \cdot
	\tilde{N})_T$ with $H_s=h I_{A_i}$ and
	$\tilde{H}_s=\tilde{h}I_{A_i}$ that $D_0( X )=\mu_{h,\tilde{h}}(A_i,A_i) = \E{\int_0^T I_{A_i}(s) R_{h,\tilde{h}}(s)\td s}$ 
satisfies		\begin{align*}
D_0(X)	&>\E{\int_0^T I_{A_i}(s)g_\lambda(h,\tilde h) \td s} = \E{\int_0^T g_\lambda(H_s,\tilde H_s) \td s} = 
\bar D_0^{\lambda}(X),
	\end{align*}
	which is in contradiction with the fact that $D$ is $\lambda$-dominated. 

Since $\mathcal{P}^n$ is an increasing sequence of $\sigma$-algebras with
$\cup_{n=1}^\infty \mathcal{P}^n=\mathcal{P}$ it follows from the martingale
convergence theorem that
$ R^{n}_{h,\tilde{h}}(t,\omega)={\it E}_\eta[R_{h,\tilde{h}}|\mathcal{P}^n](t,\omega)$ converges to 
${\it E}_\eta[R_{h,\tilde{h}}|\mathcal{P}](t,\omega)=R_{h,\tilde{h}}(t,\omega)$ 
for $\td\mathbb P\times \td t$ a.e. $(t,\omega).$
This convergence only holds up to a zero set. On this zero set, we may set $R_{h,\tilde{h}}(t,\omega)$ 
equal to $\limsup_n R^n_{h,\tilde{h}}(t,\omega)$.
Hence, this version of $R_{h,\tilde{h}}$ is dominated by $g_\lambda$ and is convex and positively 
homogeneous in $(h,\tilde{h})$ for every $(t,\omega)\in[0,T]\times\Omega$ as the limit of
convex and positively homogeneous functions.
The asserted continuity follows since every convex function that is locally bounded is continuous (see Theorem 2.2.9 in Zalinescu (2002)).
\end{proof}
\begin{proof}{\ of Lemma~\ref{lemmaeq}}
We split the proof in four steps.

\noindent\textbf{Step 1:} For $X = ((hI_{C_1})\cdot W)_T + ((\tilde{h}I_{C_2})\cdot
\tilde{N}_T$ for $(h,\tilde h)\in\mathbb R\times L^2(\nu(\td x))$ and $C_1, C_2\in\mathcal P$, 
we find by using $g(t,\omega,0,0)=0$ 
 that $D^g_0\left(X\right) = 
\E{\int_0^Tg(s,hI_{C_1}(s),\tilde{h}I_{C_2}(s))\td s}$
is equal to
\begin{multline}
\E{\int_0^T I_{C_1\setminus C_2}(s)g(s,h,0)\td s} + \E{\int_0^T
I_{C_2\setminus C_1}(s)g(s,0,\tilde{h})\td s} + \E{\int_0^T I_{C_1\cap
C_2}(s)g(s,h,\tilde{h})\td s}\\
= \mu_{h,\tilde{h}}(C_1\setminus C_2,\emptyset) +
\mu_{h,\tilde{h}}(\emptyset,C_2\setminus C_1) + \mu_{h,\tilde{h}}(C_1\cap C_2,
C_1 \cap C_2),
\label{eqimp}
\end{multline}
which is by (\ref{decompose}) equal to $\mu_{h,\tilde{h}}(C_1 , C_2) = D_0\left(X\right)$
(note that we only have to integrate over $C_1\cup C_2$ as $g(t,\omega,0,0)=0$). 

\noindent\textbf{Step 2:} Fix $t_i, t_{i+1}\in[0,T]$ with $t_i<t_{i+1}$ and
let $X = \left( (h_i I_{(t_i, t_{i+1}]})\cdot
W\right)_{t_i,t_{i+1}} + \left( ( \tilde{h}_i I_{(t_i,
t_{i+1}]})\cdot\tilde{N} \right)_{t_i,t_{i+1}}$
with $h_{i}:=\sum_{j=1}^{m} c_j I_{A_j}, \quad \tilde{h}_{i} = \sum_{j=1}^m
\tilde{c}_j I_{A_j}$,
and
$c_j\in\mathbb{R}$, $\tilde{c}_j \in L^2(\nu(\td x))$, and disjoint sets $A_j\in
\F_{t_i}$, $j=1,\ldots,m$, satisfying $\cup_j A_j=\Omega$ (we may assume w.l.o.g. 
that the $A_j$ are the same for $h$
and $\tilde{h}$ by setting some $c_j$ and $\tilde{c}_j$ equal to zero).
By step 1, denoting  $\Delta W_{{i+1}}=W_{t_{i+1}}-W_{t_i}$,
\begin{align*}
	\E{\int_{t_i}^{t_{i+1}} g(s, I_{A_j} c_j,
		I_{A_j} \tilde{c}_j)\td s} &= D_{0}\left(I_{A_j}c_j \Delta W_{{i+1}} + \IR
	I_{A_j}\tilde{c}_{j}(x)\tilde{N}((t_i,t_{i+1}]\times \td x)\right) 
	\\&= \E{D_{t_i}\left(I_{A_j}c_j \Delta W_{{i+1}} +\IR
		I_{A_j}\tilde{c}_{j}(x)\tilde{N}((t_i,t_{i+1}]\times \td x)\right)}.
\end{align*}
 Hence by Proposition~\ref{Lemmaadditive} $D_0(X)$  is equal to
\begin{eqnarray*}
\sum_{j=1}^m \E{D_{t_i}\left(I_{A_j}c_j \Delta W_{t_{i+1}} +\IR
I_{A_j}\tilde{c}_{j}(x)\tilde{N}((t_i,t_{i+1}]\times \td x)\right)}
= \E{\sum_{j=1}^m\int_{t_i}^{t_{i+1}} g(s, I_{A_j} c_j,
I_{A_j} \tilde{c}_j)\td s},
\end{eqnarray*}
which is equal to $\E{\int_{0}^{T}g(s, h_s,
\tilde{h}_s)\td s\bigg|\F_t} = {D}^g_0(X)$.

\noindent\textbf{Step 3:} Let $0\leq t_1 < \ldots<t_n=T$ be given.
For simple functions 
$X = \left( (\sum_{i=1}^l h_i I_{(t_i, t_{i+1}]})\cdot W\right)_T + \left(
(\sum_{i=1}^l \tilde{h}_i I_{(t_i, t_{i+1}]})\cdot\tilde{N} \right)_T$ for
$l\in \mathbb{N}$, with $h_i$ and $\tilde{h}_{i}$ as in step 2 we have 
by Proposition~\ref{Lemmaadditive}, step 2 
and Proposition~\ref{prop:charac}
\begin{align*}
D_0(X) &= \sum_{i=1}^l \E{D_{t_i}\left(\left( (h_i I_{(t_i, t_{i+1}]})\cdot
W\right)_{t_i,t_{i+1}} + \left( ( \tilde{h}_i I_{(t_i,
t_{i+1}]})\cdot\tilde{N} \right)_{t_i,t_{i+1}}\right)}\\
&= \sum_{i=1}^l \E{{D}^g_{t_i}\left(\left( (h_i I_{(t_i, t_{i+1}]})\cdot
W\right)_{t_i,t_{i+1}} + \left( ( \tilde{h}_i I_{(t_i,
t_{i+1}]})\cdot\tilde{N} \right)_{t_i,t_{i+1}}\right)}= {D}^g_0(X).
\end{align*}
Hence, we have $D_0(X)={D}^g_0(X)$ for all simple functions $X$.

\noindent\textbf{Step 4:} That $D_0(X)={D}^g_0(X)$ not only for simple functions but also for general $X\in L^2(\F_T)$ follows 
by the continuity of $D^g_0$ and $D_0$ in Lemma~\ref{lemmacont} (note that $g$ is of linear growth) 
and Remark~\ref{rem1}(iii).
\end{proof}

\section{Representation results, $m$-stability and time-consistency}
\label{subsec42}
We next turn to a dual representation result for general dynamic deviation measures
which is, as we show in Theorem~\ref{theoremdual1}, given in terms of {\em additively $m$-stable} representing sets (see Definition~\ref{def:mstable}). Specifically, we show that additive $m$-stability is in some sense necessary and sufficient to obtain the time-consistency axiom (D6)---see Proposition~\ref{theoremdual2}.  The proof of these results rests on auxiliary dual representation results. Using these results we first establish in Theorem~\ref{Nondom} that an integral representation of the form \eqref{charac} holds for any dynamic deviation measure even if the domination condition is not satisfied. 

In particular, we may strengthen the characterisation of dynamic deviation measures given in Theorem~\ref{Main} as follows:

\begin{theorem}\label{Nondom}
Let $D=(D_t)_{t\in[0,T]}$ be a collection of maps $D_t:L^2(\mathcal F_T)\to L^0(\mathcal F_t)$, $t\in[0,T]$.
Then $D$ is
a  dynamic deviation measure if and only if there exists a 
convex positively homogeneous driver function $g$ such that for any $t\in[0,T]$ and $X\in L^2(\mathcal F_T)$ 
\begin{equation}\label{nond-rep}
D_t(X) = \E{\int_t^T g(s,H^X_s, \tilde H^X_s)\td s\Bigg|\mathcal F_t}
\end{equation}
and $\E{\int_0^T g(s,H^X_s, \tilde H^X_s)^2\td s}<\infty$.
\end{theorem}
The mentioned notion of additive $m$-stability is the requirement of stability under additive pasting 
of subsets of the collections of (conditionally) zero-mean random variables given by
$$\mathcal{Q}_{\F_t}:=\{\xi \in L^2(\F_T)|\E{\xi|\F_t}=0\},\quad
\mathcal{Q}:=\mathcal{Q}_{\F_0}=\{\xi \in L^2(\F_T)|\E{\xi}=0\}.$$
\begin{definition}\label{def:mstable}
	A set $\mathcal{S}\subset \mathcal{Q}$ is called {\em additively $m$-stable} if for any 
	$\xi^1,\xi^2\in \mathcal{S}$ and $t\in[0,T]$, $\xi^2 + \E{\xi^1 - \xi^2|\F_t}$
	 defines an element of $\mathcal{S}$. 
\end{definition}
Denoting for a given set $\mathcal S\subset\mathcal Q$   
$$\mathcal{S}_{s,t}:=\{\E{\xi|\F_t}-\E{\xi|\F_s}|\xi\in\mathcal{S}\},\q s,t\in[0,T],$$ 
we note that $\mathcal S = \mathcal S_{0,T}$ and
 that a necessary and sufficient condition for $\S$ to be additively $m$-stable is 
$$\S=\S_{0,t}+\S_{t,T},\q  \text{for any $t\in [0,T],$}
$$ 
where $A+B$ denotes the direct sum of the sets $A$ and $B$.\footnote{That is, $A+B:=\{a+b: a\in A, b\in B\}$} 

\begin{theorem}\label{mstable}
Let $D=(D_t)_{t\in[0,T]}$ be a collection of maps $D_t:L^2(\mathcal F_T)\to L^0(\mathcal F_t)$, $t\in[0,T]$, 
satisfying (D4). Then  $D$ is a dynamic deviation measure if and only if for some convex, bounded, closed 
subset $\mathcal S^D$ of $\mathcal Q$ that contains zero and is additively $m$-stable we have
	\begin{eqnarray}
	\label{cr}
	D_t(X)&=&\esssup_{\xi\in \mathcal{S}^D\cap \mathcal{Q}_{\F_t}}\E{\xi X|\F_t}, \quad t\in[0,T].
\end{eqnarray}
\end{theorem} 

In the next result we call a $\mathcal{P}\otimes \mathcal{B}(\mathbb{R}^d)\otimes \mathcal{U}$-measurable subset $C=(C_t)_{t\in[0,T]}$ 
of $[0,T]\times \Omega\times \mathbb{R}^d\times L^2(\nu(\td x))$ closed, convex or non-empty 
if for $\td\mathbb P\times \td t$ a.e.  $(t,\omega)\in[0,T]\times\Omega,$
the sets $C_t(\omega)$ are closed, convex or non-empty. and we denote by $\mathrm{int}(C)$ 
the collection of interiors of the sets $C_t(\omega)$, $(t,\omega)\in[0,T]\times\Omega$.
\begin{theorem} \label{theoremdual1}
Let $D=(D_t)_{t\in[0,T]}$ be a collection of maps $D_t:L^2(\mathcal F_T)\to L^0(\mathcal F_t)$, $t\in[0,T]$.
Then $D$ is a 
dynamic deviation measure if and only if there exists a $\mathcal{P}\otimes
\mathcal{B}(\mathbb{R}^d)\otimes \mathcal{U}$-measurable set $C^D=(C^D_t)_{t\in[0,T]}$ that is convex, closed
with $0\in\mathrm{int}(C)$, such that  $D$ satisfies 
	the representation in \eqref{cr} with a bounded set $\mathcal S^D$ given in terms of $C^D$ by
\begin{eqnarray}
	\label{S}
\mathcal{S}^D&=&\Big\{\xi \in\mathcal{Q}\bigg| (H^\xi_t,\tilde{H}^\xi_t)\in C^D_t\mbox{ for all }t\in[0,T]\Big\}.
	\end{eqnarray}
\end{theorem}
The proofs of Theorems \ref{Nondom}, \ref{mstable} and \ref{theoremdual1}
are given below.
\smallskip

\begin{remark}[Relation to strong time-consistency of dynamic risk-measures]{\rm 
The characterisation in Theorem \ref{theoremdual1} is reminiscent of analogous characterisation 
results of (strong) time-consistency of dynamic risk measures 
available in the literature. If we call a set $\mathcal{S}'\subset \mathcal{M}$ {\em multiplicatively $m$-stable}  if
	for every $\xi^1,\xi^2\in \mathcal{S}'$ and $t\in[0,T]$
	the element  $L_t := \xi^1_t\xi^2_T/\xi^2_t$ is contained in $\mathcal{S}',$
we note that under multiplicative $m$-stability of $\S'$ we have the decomposition
$\S'=\S'_{0,T}=\S'_{0,t}\S'_{t,T}$ with  $\mathcal{S}'_{s,t}:=\{\E{\xi|\F_t}/\E{\xi|\F_s}|\xi\in\mathcal{S}'\}$ (with $0/0=0$),
so that the set $\S'$ is stable under `multiplicative' pasting.
It is well-known that coherent risk measures are (strongly) time-consistent 
precisely if the representing sets in the corresponding dual representations are 
multiplicatively $m$-stable; see among many others Chen and Epstein (2002) (where multiplicative $m$-stablility is called `rectangular property'), Riedel (2004), Delbaen~(2006), Artzner {\em et al.}~(2007) or  F\"ollmer and Schied~(2011). 
Specifically, in a 
Brownian setting it is shown in Delbaen (2006) that multiplicative $m$-stability of 
a convex and closed set $\mathcal{S}'\subset \mathcal{M}:=\{\xi \in L^1_+(\F_T)|\E{\xi}=1\}$ 
 containing $1$ corresponds to the existence of a $\mathcal{P}\otimes
	\mathcal{B}(\mathbb{R}^d)\otimes \mathcal{U}$-measurable, closed and convex set $C'$ containing $0$ such that  
$\mathcal{S}'=\{\xi\in\mathcal{M}|(q^\xi_s,\psi^\xi_s)\in C'_s\mbox{ for all }s\in[0,T]\}$, where $(q^\xi,\psi^\xi)$ is  
related to the stochastic logarithm of $\xi$ by
	$\xi=\mathcal{E}\Big((q^\xi\cdot W)_T+(\psi^\xi\cdot \tilde{N})_T\Big)$ with $\mathcal{E}(\cdot)$ 
	denoting the Dol\'eans-Dade exponential. This result implies that time-consistent coherent risk measures on $L^\infty$ satisfy the representation
\begin{eqnarray}
\label{delbaen}
\p_t(X)&=&\esssup_{\xi \in\S'\cap \mathcal{M}_{\F_t}}\E{-\xi X|\F_t},\ \mbox{with $\mathcal{M}_{\F_t}:=\{\xi\in L^1_+(\F_T)|\E{\xi|\F_t}=1\}$ and}\nonumber\\
\S'&=&\Big \{\xi\in L^1_+(\F_T)|(q^\xi_s,\psi^\xi_s)\in C'_s\mbox{ for all }s\in[0,T]\Big\}.
\end{eqnarray}
This result is generalized in Delbaen {\em et al.} (2010) to convex risk measures. 
As a counterpart of Theorem 3.1 in Delbaen (2006), which concerns multiplicatively $m$-stable sets
in a Brownian filtration, we have from Theorem \ref{theoremdual1} and Propositions~\ref{prop2}--\ref{theoremdual2} below 
that 
{\em a closed and convex set $\S\subset \mathcal{Q}$ containing $0$ is additively $m$-stable if and only if, for some $\mathcal{P}\otimes
\mathcal{B}(\mathbb{R}^d)\otimes \mathcal{U}$-measurable set $C^*=(C^*_t)_{t\in[0,T]}$ 
that is convex, closed and contains $0$, we have
$\mathcal{S}=\{\xi \in\mathcal{Q}| (H^\xi_t,\tilde{H}^\xi_t)\in C^*_t\mbox{ for all }t\in[0,T]\}.
$
} 
}
\end{remark}

\noindent{\bf Auxiliary representation results.} Our starting point is the $\F_t$-conditional version of the duality result given in Theorem 1 in Rockafellar {\em et al.} (2006a). 
\begin{proposition}
	\label{prop2}
	Let $t\in [0,T]$ and let the map $D_t:L^2(\mathcal F_T)\to L^0(\mathcal F_t)$ be given.

\noindent{\bf (i)} $D_t$ satisfies (D1)-(D3) and (D5) and maps $L^2(\F_T)$ to $L_+^2(\F_t)$ 
if and only if there exists a bounded, closed and convex set $\mathcal{S}_{D_t}\subset
	\mathcal{Q}_{\F_t}$ containing zero such
	that \be
	D_t(X)=\esssup_{\xi\in \mathcal{S}_{D_t}}\E{\xi X|\F_t},\q X\in L^2(\F_T).
	\label{Dual1}
	\ee
The set	$\S_{D_t}$ is uniquely determined by its (convex) indicator function $J_{\S_{D_t}}:L^2(\F_T)\to\{0,\infty\}$ given by
\be
	\label{Dual2}
	J_{\S_{D_t}}(\xi):=\esssup_{X\in L^2(\F_T)}\{\E{\xi X|\F_t}-D_t(X)\}.\ee

\noindent{\bf (ii)}	Assume the conditions in (i) are satisfied. Then $D_t$ satisfies (D4) 
if and only if for every $X\in L^2(\F_T)$ with $X\notin L^2(\F_t)$ there 
	exists $\xi\in \S_{D_t}$ such that $\mathbb P[\E{\xi X|\F_t}>0]>0$.
\end{proposition}
\begin{Remark}\label{rd}
Note that by (\ref{Dual2}) we have for any set $A\in\F_t$ and $\xi^1,\xi^2\in \S_{D_t}$  that $I_A \xi^1+I_{A^c} \xi^2\in\S_{D_t}.$ Sets having this property are \emph{directed}.\footnote{A set $\S$ is called directed if for any $\xi^1,\xi^2\in\S$ there exists $\bar{\xi}\in\S$ with $\bar{\xi}\geq \xi^1\vee \xi^2$.}
\end{Remark}

Hence, $D_t(X)$ admits a robust representation with representing set given by a collection of signed measures.
This proposition is stated in Rockafellar {\em et al.} (2006a) in a static setting but it can be seen to also hold true
conditionally on $\mathcal F_t$---see for instance Riedel (2004), Ruszczy\'nski and Shapiro (2006), or Cheridito and Kupper (2011) for related arguments.

For dynamic deviation measures the property (D6) induces a specific structure of 
the sets $\mathcal S_{D_t}$, $t\in[0,T]$, which we specify in the next results.
A first observation is as follows:
\begin{proposition}
\label{propmastable}Let $t\in [0,T]$ and let $D$ be a dynamic deviation measure and denote
$\S^D:=\S_{D_T}$. We have that the set $\mathcal S_{D_t}$ in the representation (\ref{Dual1}) 
of $D_t$ is such that  $\S_{D_t}=\S^D\cap \mathcal{Q}_{\F_t}=\S^D_{t,T}$.
\end{proposition}
\noindent{\it Proof of Proposition \ref{propmastable}.}
Let $\xi\in L^2(\F_T)$ and $t\in[0,T]$. For brevity we denote throughout the proof 
$\S = \S^D$, $\S_t = \S_{D_t}$ and $\S_{t,T} = \S^D_{t,T}$.
As it is clear that 
$\S\cap \Q_{\F_t}=\S_{t,T}$ (noting that $\S_{t,T}\subset \Q_{\F_t}$), the remainder 
of the proof is concerned with showing that the sets $\S\cap \Q_{\F_t}$ and $\S_{t}$ are equal.

Noting that $\E{D_t(X)} \leq D_0(X)$ (by (D6)), recalling (\ref{Dual2}) and deploying (D6), (D1) and the 
fact that $L^2(\F_T)$ is directed,
we have for $\xi\in\S_{t}\subset \mathcal{Q}_{\F_t}$
\begin{align*}
J_{\S}(\xi)&=\sup_{X\in L^2(\F_T)}\{\E{\xi X}-D_0(X)\}
\leq
\sup_{X\in L^2(\F_T)}\{\E{\xi X}-\E{D_t(X)}\}\\
&=
\sup_{X\in L^2(\F_T)}\E{\E{\xi X|\F_t}-D_t(X)}=
\E{\esssup_{X\in L^2(\F_T)}\{\E{\xi X|\F_t}-D_t(X)\}}=0,
\end{align*}
where in the last equality we used (\ref{Dual2}).
As $J_{\S}(\xi)$ is either zero or infinity it follows from the previous display 
that $J_{\S}(\xi)=0$ implying that $\xi \in \S$ 
and thus $\xi\in \S\cap \mathcal{Q}_{\F_t}$. This shows $\S_{t}\subset \S\cap \mathcal{Q}_{\F_t}$.

On the other hand, if $\xi\in \S^c_{t}:=L^2(\F_T)\backslash\S_t$ 
then we have either {(a)} $\xi \in (L^2(\F_T)\backslash\mathcal Q_{\F_t})\cap\mathcal S^c_t$ 
or {(b)} $\xi \in \mathcal{Q}_{\F_t}\cap\mathcal S^c_t$. In case {(a)} we have $\xi \notin 
\S\cap \mathcal{Q}_{\F_t},$ while in case {(b)} \eqref{Dual2} in Proposition~\ref{prop2} yields that 
there exists $X'\in L^2(\F_T)$ such that $\E{\xi X'|\F_t}-D_t(X')> 0$ 
 on a non-zero set, say $A$. Hence by using (D6) and that $\xi\in\Q_{\F_t}$ we have
(from \eqref{Dual2} with $t=0$)
\begin{align*}	
J_{\S}(\xi) &\geq \E{\xi I_A X'}-\E{D_t(X'I_A)}\\
&= \E{I_A(\xi X'-D_t(X'))}=\E{I_A(\E{\xi X'|\F_t}-D_t(X'))}>0.
\end{align*}
Thus, $J_{\S}(\xi)=\infty$ and we have that $\xi\notin \S\cap \mathcal Q_{\F_t}$, also in case {(b)}. 
Hence, $\S^c_t\subset L^2(\F_T)\backslash(\S\cap \mathcal{Q}_{\F_t})$. 
Combined with the inclusion derived in previous paragraph this yields that $\S_t=\S\cap \mathcal{Q}_{\F_t}$.
\qed

The following result shows that stability under `additive pasting' of the representing set 
in the form of additive $m$-stability is  a necessary and sufficient condition for  (D6) to hold.
\begin{proposition}
	\label{theoremdual2} Let $\S\subset \mathcal{Q}$ be 
a convex, closed set  containing zero. $\S$ is additively $m$-stable if and only if
the collection $D_t(X):=\esssup_{\xi\in \S\cap \mathcal{Q}_{\F_t}}\E{\xi X|\F_t}$, $t\in [0,T]$, $X\in L^2(\F_T)$,
satisfies (D6).
\end{proposition}
\begin{proof}{}
	We first show `$\Rightarrow$'. We only give the proof that (D6) holds for $s=0$ as the proof for $s\in (0,T]$ 
	is analogous. Let $X\in L^2(\F_T)$ and $t\in[0,T]$. Denoting $\xi_t=\E{\xi|\F_t}$ and $\xi_{t,T}=\xi-\xi_t$ for $\xi\in L^2(\F_T)$
	we have 
	\begin{align*}
	D_0(X)&=\sup_{\xi\in \S}\E{\xi X} =\sup_{\xi\in \S}\E{\E{\xi_tX +(\xi-\xi_t) X|\F_t}}\\
	&=\sup_{\xi=\xi_t+\xi_{t,T}\in \S_{0,t}+\S_{t,T}}\{\E{\xi_tX }+\E{\E{\xi_{t,T} X|\F_t}}\}
	=\sup_{\xi_t\in \S_{0,t},\xi_{t,T}\in \S_{t,T}}\{\E{\xi_tX }+\E{\E{\xi_{t,T} X|\F_t}}\}.
	\end{align*}
	Hence by the directedness of $\S_{t,T}$ (Remark~\ref{rd}) and Proposition~\ref{propmastable} we obtain  
	\begin{align*}
	D_0(X)
	&=\sup_{\xi_t\in \S_{0,t}}\E{\xi_t\E{X|\F_t} }+\sup_{\xi_{t,T}\in \S_{t,T}}\E{\E{\xi_{t,T} X|\F_t}}\\
	&=\sup_{\xi\in \S}\E{\xi\E{X|\F_t} }
	+\E{\esssup_{\xi_{t,T}\in \S_{t,T}}\E{\xi_{t,T} X|\F_t}}\\
	&=D_0(\E{X|\F_t})+\E{\esssup_{\xi\in \S\cap \mathcal{Q}_{\F_t}}\E{\xi X|\F_t}}=D_0(\E{X|\F_t})+\E{D_t(X)}.
	\end{align*}
	
To see that we have `$\Leftarrow$' suppose that $\xi^1,\xi^2\in\mathcal{S}$ such that $\xi^1_t+(\xi^2-\xi^2_t)\notin \S$ for 
some $t\in [0,T]$. Then by the Hahn-Banach Theorem there exists a random variable $X\in L^2(\F_T)$ such that we have
	\be
	\label{HB}
E:=	\E{(\xi^1_t+(\xi^2-\xi^2_t))X}>\sup_{\xi\in \S}\E{\xi X}=D_0(X) .
	\ee
Using Proposition \ref{propmastable} we note
	$E  = \E{\xi^1_t\E{X|\F_t}} +\E{\E{(\xi^2-\xi^2_t)X|\F_t}}$ 
	may be bounded above by
$$D_0(\E{X|\F_t})+\E{\esssup_{\xi\in \S_{t,T}}\E{\xi X|\F_t}}= D_0(\E{X|\F_t})+\E{D_t(X)}=D_0(X).$$
	This bound is a contradiction to (\ref{HB}), which proves the implication `$\Leftarrow$'.
\end{proof}

\begin{proof}{~of Theorem~\ref{mstable}}
The assertion follows by combining Propositions \ref{prop2}, \ref{propmastable} and \ref{theoremdual2}.
\end{proof}
In the proofs of Theorems~\ref{Nondom} and \ref{theoremdual2} we deploy, for a given dynamic deviation measure $D$, the sequence $(D^{(n)})_{n\in\mathbb N}$  
of dynamic deviation measures $D^{(n)}=(D^{(n)}_t)_{t\in[0,T]}$, $D_t^{(n)}:L^2(\F_T)\to L^2(\F_t)$
defined by  
\begin{eqnarray}
\label{defdk}
D^{(n)}_t(X)&:=&\esssup_{\xi\in (\S^D\cap \mathcal{Q}_{\F_t})\cap \mathcal{A}^n}\E{\xi X|\F_t},\ \text{with}\\
\mathcal{A}^n&:=&\left\{\xi\in L^2(\mathcal F_T) \left| \sup_{s\in[0,T]}\left\{|H^\xi_s|^2+\IR|\tilde{H}^\xi_s(x)|^2\nu(\td x)\right.\right\}\leq n^2\right\}.\label{defak}
\end{eqnarray}
\begin{Lemma}\label{lemmadk} 
Let $t\in[0,T]$ and $X\in L^2(\F_T)$ and, for a given dynamic deviation measure $D$, 
let $(D^{(n)})_{n\in\mathbb N}$ and $(\mathcal A^n)_{n\in\mathbb N}$ be as in 
\eqref{defdk}--\eqref{defak}.

\noindent{\bf (i)} for any $n\in\mathbb N$, we have $D^{(n)}_t(X)\leq D^{(n+1)}_t(X)$ 
and $\A^{n+1}=\frac{n+1}{n} \A^n$. Moreover, $D^{(n)}_t(X)\nearrow D_t(X)$ in $L^2(\F_t)$ as $n\to\infty$.\\
\noindent{\bf (ii)} for any $n\in\mathbb N$, $\S\cap\mathcal{A}^n$ contains zero and is closed, bounded, 
convex and additively $m$-stable.\\
\noindent{\bf (iii)}
For any $n\in\mathbb{N}$, $D^{(n)}$ is a  dynamic deviation measure that is $n$-dominated.
\end{Lemma}
\begin{proof}{}\ {\bf (i)}
It is easily verified that $\mathcal A^{n+1}=\frac{n+1}{n}\A^n$ so that $\A^n\subset \A^{n+1}$ for $n\in\mathbb N$.
Hence, by \eqref{defdk} we have $D^{(n)}_t(X)\leq D^{(n+1)}_t(X)$ for $t\in[0,T]$ and $X\in L^2(\mathcal F_T)$.
Furthermore, as $(\mathcal{A}^n)_{n\in\mathbb N}$ is dense in $L^2(\F_T)$ and the set $\mathcal S^D$ in Theorem~\ref{mstable} is bounded, we have that $D_t^{(n)}(X)\nearrow D_t(X)$ as $n\to\infty$.\\
\noindent{\bf (ii)} Let $n\in\mathbb N$.
It is straightforward to verify that $\mathcal{A}^n$  contains zero and is closed, bounded and convex.
	Let us show next that $\mathcal{A}^n$ is additively $m$-stable. Let $t\in[0,T]$ and $\xi^1,\xi^2\in\mathcal{A}^n$ 
and denote $L= \xi^2 + \E{\xi^1-\xi^2|\mathcal F_t}.$ Then 
the representing pair $(H^{L},\tilde H^{L})$ of $L\in L^2(\mathcal F_T)$ is expressed in terms of the representing pairs
$(H^{i},\tilde H^{i})$, $i=1,2$, of $\xi^1$, $\xi^2$ by
$H^{L}_s = H^1_s I_{[0,t]}(s) + H^2_s I_{(t,T]}(s)$ and $\tilde H^{L}_s = \tilde H^1_s I_{[0,t]}(s) 
+ \tilde H^2_s I_{(t,T]}(s)$. In particular,  we have $\sup_{s\in[0,T]}\big\{|H^{L}_s|^2+\IR|\tilde{H}^{L}_s(x)|^2\nu(\td x)\}\leq n^2$ 
so that $L\in \mathcal{A}^n$. Thus, $\A^n$ is additively $m$-stable.
Since the set $\S^D$ is also closed, convex and additively $m$-stable, the same holds for 
$\mathcal{A}^n\cap\S$. 
 
 \noindent{\bf (iii)} Let $n\in\mathbb N$.
From Proposition \ref{prop2} and part (ii) we conclude that $D^{(n)}$ satisfies (D1)--(D3) and (D5).  Furthermore, from Proposition~\ref{theoremdual2} and part (ii) we have that $D^{(n)}$ satisfies (D6).
Let us show next that $D^{(n)}$ satisfies positivity (D4).	Let $t\in[0,T]$ and $X\in L^2(\F_T)\backslash L^2(\F_t)$. By Propositions~\ref{prop2} and \ref{propmastable}  there exists a $\tilde{\xi}\in \S^D\cap \mathcal{Q}_{\F_t}$ 
such that $\E{\tilde{\xi} X|\F_t}>0$ on a non-zero set.  As $(\A^n)_{n\in\mathbb N}$ is increasing and dense in $L^2(\F_T)$ (as noted in the proof of part (i)), we can find a sequence $(\xi^m)_m$ such that $\xi^{m}\in \mathcal S^D\cap \mathcal{Q}_{\F_t}\cap \A^m$ converges to $\tilde{\xi}$ in $L^2(\F_T)$ as $m\to\infty$. Next, choose 
$m'$ sufficiently large such that on a non-zero set, say $A$, we have $\E{\xi^{m'} X|\F_t}>0$ (which is possible since $\xi^{m}X$ converges to $\tilde{\xi} X$ in $L^1$ as $m\to\infty$).  
Define $\xi^*\in \mathcal S^D\cap\mathcal Q_{\F_t}\cap \A^n$ by 
$ \xi^*:=\frac{n}{m'}\xi^{m'}.$
Since on $A$ we have $\E{\xi^* X|\F_t}=\frac{n}{m'}\E{\tilde{\xi}^{m'} X|\F_t}>0\,$
we conclude from~\eqref{defdk} that $D^{(n)}$ satisfies (D4).

Finally, by deploying the Cauchy-Schwarz inequality 
we note that $D^{(n)}_t(X)$ 
may be bounded above by   
	\begin{multline}
\sup_{\xi\in \mathcal{A}^n}\E{\xi X|\F_t}
	=	\sup_{\xi\in \mathcal{A}^n}\E{\int_0^T \left.\left( 
	(H^\xi_s)^{\intercal} H^X_s+\IR \tilde{H}^\xi_s(x)\tilde{H}^X_s(x)\nu(\td x)\right)\td s \right|\F_t }
\\
	\leq n\,\E{\left.\int_0^T \sqrt{|H^X_s |^2+\IR |\tilde{H}^X_s(x)|^2\nu(\td x)}\td s\right|\F_t } = 
	\bar D_t^{n}(X), \label{CS}
	\end{multline}
	where we denote by $v^{\intercal}$ the transpose of the column vector $v\in\mathbb R^d$.
\end{proof}

\subsection{Proof of Theorem~\ref{Nondom}}\label{p:nondom}
With the previously established results in hand we can now complete the proof of Theorem~\ref{Nondom}.
As the arguments in the proof of the implication `$\Leftarrow$' in Theorem~\ref{Main} carry over for the proof of `$\Leftarrow$'  in Theorem~\ref{Nondom}, the remainder of the proof is concerned with the proof of `$\Rightarrow$'.
Let $D$ be a dynamic deviation measure, $X\in L^2(\F_T)$ and denote by $(D^{(n)})_{n\in\mathbb N}$ the approximating sequence of 
dynamic deviation measures from Lemma~\ref{lemmadk}.  By Lemma \ref{lemmadk}(i,iii) and Theorem~\ref{Main}  
the sequence $(D^{(n)}(X)_{n\in\mathbb N}$ is monotone increasing and there exists a sequence $(g^n)_{n\in\mathbb N}$ of convex and positively homogeneous driver functions  such that \eqref{nond-rep} holds (with $D$ and $g$ replaced by $D^{(n)}$ and $g^n$). Therefore, by Proposition~\ref{theoremequivalent}(iv), $g^n\leq g^{n+1}$ for $n\in\mathbb N$, so that we can define 
$g:=\lim_{n\to\infty}g^n.$ Clearly, $g$ is convex, positively homogeneous and lower semi-continuous as the limit of functions having these properties. Furthermore, for $(h,\tilde{h})\ne 0$ we have $g(\omega,t,h,\tilde{h})\geq g^1(\omega,t,h,\tilde{h})>0$ and $g(\omega,t,0,0)=\lim_{n\to\infty}g^n(\omega,t,0,0)=0$ $\td\mathbb P\times \td t$ a.e. 
Hence, $g$ is a convex and positively homogeneous driver function.
Finally, as $(g^n)_n$ is an increasing sequence of functions
an application of the monotone convergence theorem yields
\begin{align*}
D_t(X)&=\lim_n D^{(n)}_t(X)=\lim_n \E{\int_t^T  g^n(s, H^{X}_s, \tilde{H}^{X}_s)\td s\Bigg|\mathcal F_t}=\E{\int_t^T  g(s, H^{X}_s, \tilde{H}^{X}_s)\td s\Bigg|\mathcal F_t}.
\end{align*}
This completes the proof of Theorem~\ref{Nondom}.

\subsection{Proof of Theorem~\ref{theoremdual1}}\label{s6:p}
In the proof of Theorem \ref{theoremdual1} we deploy the following auxiliary result:
\begin{lemma}
	\label{lemmaint}
\noindent{\bf (i)} Let $g$ be a convex and positively homogeneous driver function and let
the $\mathcal{P}\otimes
\mathcal{B}(\mathbb{R}^d)\otimes \mathcal{U}$-measurable set $C=(C_t)_{t\in[0,T]}$ be determined by 
$$J_{C_t}(u,\tilde{u})=r(t,u,\tilde{u}):=\sup_{u\in\mathbb{R}^d,\tilde{u}\in L^2(\nu(\td x))}\{u^{\intercal} h+\IR \tilde{u}(x)\tilde{h}(x)\nu(\td x)-g(t,u,\tilde{u})\}$$
for $u\in\mathbb R^d$ and $\tilde u\in L^2(\nu(\td x))$.
Then $0\in\mathrm{int}(C_t)(\omega)$ $\td\mathbb P\times \td t$ a.e. 

\noindent{\bf (ii)} Let $C^D=(C^D_t)_{t\in[0,T]}$ be a $\mathcal{P}\otimes
\mathcal{B}(\mathbb{R}^d)\otimes \mathcal{U}$-measurable set and let 
$\S^D$ be given by the right-hand side of \eqref{S}. If $0\in\mathrm{int}(C^D_t)(\omega)$ $\td\mathbb P\times \td t$ a.e. 
then, for any $t\in[0,T]$ and $X\in L^2(\F_T)\backslash L^2(\F_t)$, there exists a $\xi'\in\mathcal S^D$ 
such that $\mathbb P(\E{\xi'X|\F_t }>0)>0$.
\end{lemma}
\begin{proof}{}\ 
	To simplify notation we denote $z:=(h,\tilde{h})$ and $y=(q,\psi)$ 
	for elements $(h,\tilde{h}), (q,\psi)$ in the Hilbert space $\mathbb{R}^d\times L^2(\nu(\td x))$. Further we denote $\langle y,z\rangle_*=qh+\IR
	\psi(x)\tilde{h}(x)\nu(\td x)$ and $|z|_*=\sqrt{|h|^2+\IR
	|\tilde{h}(x)|^2\nu(\td x)}$. 
	
	\noindent{\bf (i)} Set $\mathcal{Z}:=\{z\in \mathbb{R}^d\times L^2(\nu(\td x))| |z|_*=1 \}$ 
	and for $z\in\mathcal{Z}$ and $\lambda\in\mathbb{R}$ we denote $z^\lambda:=\lambda z.$ By the positive homogeneity of $g$ and the symmetry of the set $\mathcal{Z}$ we have for fixed $y\in \mathbb{R}^d\times L^2(\nu(\td x))$
that $r(t,y)=\sup_{z\in\mathcal{Z},\lambda\in\mathbb{R}}\{\langle y,z^\lambda\rangle_* -g(t,z^\lambda)\}$
is equal to
	\begin{align}\label{rty}
r(t,y)
	&=\sup_{z\in\mathcal{Z},\lambda\geq 0}\Big\{\langle y,z^\lambda \rangle_*-g(t,z^\lambda) \Big\}
= \sup_{z\in\mathcal{Z},\lambda\geq 0}\lambda\Big\{\langle y,z \rangle_*-g(t,z) \Big\}.	\end{align}
The supremum in \eqref{rty} is finite if and only if for all $z\in\mathcal{Z}$
	$\langle y,z\rangle_*\leq g(t,z).$
Letting $(y_n)_n$ be a sequence such that $|y_n|_*\to 0$ and using the Cauchy-Schwarz inequality, we have  that $$\sup_{z\in\mathcal{Z}}|\langle y,z\rangle_*|\leq |y_n|_* \sup_{z\in\mathcal{Z}}|z|_* =|y_n|_*\to 0 .$$
	Since by assumption $g(t,z)>0$ for every fixed $z\in \mathcal{Z}$ we have that from a certain $n$ onwards
	$\langle y_n,z\rangle_*\leq g(t,z)$ so that
	$r(t, y_n)=0$. As $r(t,\omega,y_n)=J_{C_t(\omega)}(y_n)$ this entails that $y_n\in C_t(\omega)$ from a certain $n$ onwards for every sequence $y_n$ that is such that $|y_n|_*\to 0$. Hence, $0\in\mathrm{int}(C_t(\omega))$.

\noindent{\bf (ii)} Let $t\in[0,T]$ and $X\in L^2(\F_T)\backslash L^2(\F_t)$.  
For any $s\in[0,T]$, we note that if $0\in\mathrm{int}(C^D_s(\omega))$ then there exists $\varepsilon'_s(\omega)\in(0,1]$ 
such that $|y|_*\leq \varepsilon'_s(\omega)$ implies $y\in C^D_s(\omega)$. Define $\lambda_s(\omega):=|(H^X_s(\omega),\tilde H^X_s(\omega))|^2_*$, $A = \{(s,\omega)\in[t,T]\times\Omega: \lambda_s(\omega)>0\}$ 
and denote by $\varepsilon=(\varepsilon_s)_{s\in[0,T]}$
the process given by $\varepsilon_s(\omega) := I_A(s,\omega)\varepsilon'_s(\omega)/\lambda_s(\omega)$. 
Then $\xi' := (\varepsilon\, H^X \cdot W)_{t,T} + (\varepsilon\, \tilde H^X\cdot \tilde N)_{t,T}$ is element 
of $\S^D$.
Since $X\in L^2(\F_T)\backslash L^2(\F_t)$, the set $A$ has positive $\td\mathbb P\times\td t$-measure so that 
$$\E{\E{X\xi'|\F_t}} = \E{\E{\left.\int_t^TI_A\varepsilon'_s\td s\right|\mathcal F_t}} = 
\E{\int_t^TI_A\varepsilon'_s\td s} > 0,$$
which implies, as $\E{X\xi'|\F_t}$ is nonnegative, that $\mathbb P(\E{X\xi'|\F_t}>0)>0$.
\end{proof}
\noindent\textit{Proof of Theorem \ref{theoremdual1}.}
Let us first show the implication `$\Leftarrow$': 
We note first that, as is straightforward to 
verify, $\S^D$ given in \eqref{S} is additively $m$-stable, 
convex, bounded, closed and contains zero. 
Moreover, Lemma~\ref{lemmaint} and Proposition~\ref{prop2}(ii) 
imply that, for any $t\in[0,T]$, $D_t:L^2(\F_T)\to L^2(\F_t)$ defined by \eqref{cr}
satisfies (D4). Hence, by Theorem~\ref{mstable} 
$D=(D_t)_{t\in[0,T]}$ is a dynamic deviation measure.

We next turn to the proof of `$\Rightarrow$'. In view of Theorem~\ref{mstable} it suffices to show that 
$\mathcal S^D$ is given by the expression in \eqref{S}.
For any $n\in\mathbb N$ let $D^{(n)}$ be defined as in (\ref{defdk}). As noted before
$(D^{(n)})_{n\in\mathbb N}$ is a collection of dynamic deviation measures increasing to $D$ (Lemma~\ref{lemmadk})
and the corresponding sequence $(g^n)_{n\in\mathbb N}$ of driver functions  is increasing and
satisfies $g^n\leq g$ (Proposition~\ref{theoremequivalent}(iv)), where $g$ is the function 
in the representation \eqref{nond-rep} of $D$ (in Theorem~\ref{Nondom}).
For $u\in\mathbb{R}^d$, $\tilde{u}\in L^2(\nu(\td x))$ and $n\in\mathbb N$ define
$$r^n(s,u,\tilde{u}):=\sup_{h\in\mathbb{Q}^d,\tilde{h}\in \{h_1,h_2,h_3,\ldots\}}\left\{u^{\intercal}\,h+\IR \tilde{u}(x)\tilde{h}(x)\nu(\td x)-g^n(s,h,\tilde{h})\right\},$$
where  $\{h_1,h_2,h_3,\ldots\}$ denotes a countable basis of $L^2(\nu(\td x))$.
Note that for any $n\in\mathbb{N}$ we have (i) $r^n$  lower semi-continuous and convex in $(u,\tilde{u})$ and  (ii) $r^n$ is a (convex) indicator function of some convex and closed set, say $C^n=(C^n_s)_{s\in[0,T]}$.  
Furthermore, we note the following observations: 
{(a)} since $r^n$ is the supremum of a $\mathcal{P}\otimes \mathcal{B}(\mathbb{R}^d)\otimes \mathcal{U}$ measurable process and $C^n$ is the set where $r^n$ is equal to zero, we have 
that $C^n$ is also $\mathcal{P}\otimes \mathcal{B}(\mathbb{R}^d)\otimes \mathcal{U}$-measurable and
{(b)} as the functions $g^n(s,h,\tilde h)$ is continuous in $(h,\tilde{h})$, $r^n$ coincides with the dual conjugate of $g^n$, so that  we have $\td\mathbb P\times \td t$ a.e.
\be\label{gck}
g^n(s,\omega,h,\tilde{h})=\sup_{(u,\tilde{u})\in C^n_s(\omega)}\left\{u^{\intercal}\,h+\IR \tilde{u}(x)\tilde{h}(x)\nu(\td x)\right\}.
\ee
Moreover, we have that {(c)} as the sequence $(g^n)_n$ is increasing, $(r^n)_n$ is a decreasing sequence 
so that $C^n\subset C^{n+1}$ for any $n\in\mathbb N$.
Denote $C=\cup_{n=1}^\infty C^n$ and note that $C$ is convex and measurable as the increasing union of convex and measurable  sets.

Let us next establish the representation (\ref{cr}) for $D^{(n)}(X)$ for given $n\in\mathbb N$ and $X\in L^2(\F_T)$.
As $D^{(n)}(X) = D^{g^n}(X)$ we have
\begin{align}
 D^{(n)}_0(X)&=\E{\int_0^T \sup_{(u,\tilde{u})\in C^n_s} \left(u^{\intercal} H_s^X+\IR \tilde{u}(x)\tilde{H}^X_s(x)\nu(\td x)\right)\td s}\nonumber\\
&\geq \sup_{\{(H,\tilde{H})|(H_s,\tilde{H}_s)\in C^n_s, s\in[0,T]\}}\E{\int_0^T  \left(H_s^{\intercal} H_s^X+\IR \tilde{H}_s(x)\tilde{H}^X_s(x)\nu(\td x)\right)\td s} \label{last5}
\\
&= \sup_{\xi\in \mathcal{M}^n}\E{\int_0^T  \left((H^\xi_s)^{\intercal} H_s^X+\IR \tilde{H}^\xi_s(x)\tilde{H}^X_s(x)\nu(\td x)\right)\td s},\label{last6}
\end{align}
with $\mathcal{M}^n:=\{\xi\in\mathcal{Q}|(H^\xi_s,\tilde{H}^\xi_s)\in C^k_s, s\in[0,T]\,\}$, 
where the supremum in \eqref{last5} is taken over pairs $(H,\tilde H)\in L_d^2(\mathcal{P},\td\mathbb P\times \td t)\times
L^2(\mathcal{P}\times \mathcal{B}(\R^k\setminus\{0\} ),\td\mathbb P\times \td t \times \nu(\td x))$. 
 
Let us show next that the inequality in (\ref{last5}) is in fact an equality.
It is well known (see for instance Theorem 2.4.9 in Zalinescu (2002)) that the subgradients of continuous and convex functions are non-empty so that the suprema in the dual representations of the functions $g^n$, $n\in\mathbb N$, 
are attained. Hence, we can apply a measurable selection theorem to the set
$$G^n:=\left\{(s,\omega,u,\tilde{u})\,\left| g^n(s,\omega,H^X_s,\tilde{H}^X_s)- u^{\intercal}\,H^X_s - \IR \tilde{u}(x)\tilde{H}^X_s(x)\nu(\td x) +J_{ C^n_s(\omega)}(u,\tilde{u})=0\right.\right\}, $$
obtaining $\mathcal{P}\times \mathcal{P}\otimes \mathcal{U}$-measurable processes $(U^n,\tilde{U}^n)$ 
such that, for every $s$, $(U^n_s,\tilde{U}^n_s)\in C^n_s$ and
$g^n(s,H^X_s,\tilde{H}^X_s)= (U^n_s)^{\intercal} H^X_s+\IR \tilde{U}^n_s(x)\tilde{H}^X_s(x)\nu(\td x)$. 
This implies (\ref{last5}) holds with equality, and yields the desired representation for $D^{(n)}$.

To see that we also get a representation for $D$ let us first prove that the set $C$ defined above 
(our natural candidate to satisfy \eqref{cr}--\eqref{S}) is closed.
Note that from (\ref{last6}) it follows that for any $X\in L^2(\F_T)$
\be\label{dualsets}
\sup_{\xi\in\S\cap\mathcal{A}^n}\E{\xi X}=D_0^{(n)}(X)= \sup_{\xi\in \mathcal{M}^n}\E{\int_0^T  \left((H^\xi_s)^{\intercal} H_s^X+\IR \tilde{H}^\xi_s(x)\tilde{H}^X_s(x)\nu(\td x)\right)\td s}.
	\ee
As $\S\cap \mathcal{A}^n$ and $\mathcal{M}^n$ are both convex and closed sets, we conclude from (\ref{dualsets})
$\S\cap\mathcal{A}^n=\mathcal{M}^n.$ In particular, for $m\geq n$ we have $\mathcal{M}^n=\mathcal{M}^m \cap \mathcal{A}^n$. As there is a one-to-one correspondence between  $\xi\in\mathcal{Q}$ and square-integrable predictable processes $(H,\tilde{H})$ this entails that $$C^n=C^m\cap \left\{(H,\tilde{H})\in L^2(\td\mathbb P\times \td t)\times  L^2(\td\mathbb P\times \td t\times \nu(\td x))\left|\sup_{t\in[0,T]}\{|H_t|^2+\IR |\tilde{H}_t(x)|^2 \nu(\td x)\}\leq n^2\right.\right\}.$$
Hence, $\td\mathbb P\times \td t$ a.e.
$$C^n_t(\omega)=C^m_t(\omega)\cap \left\{(h,\tilde{h})\in \mathbb{R}^d\times L^2(\nu(\td x))\left||h|^2+\IR |\tilde{h}(x)|^2 \nu(\td x)\leq n^2\right.\right\}.$$
Taking the union over all $m\in\mathbb N$ on the right-hand side of previous display yields
$$C^n_t(\omega)=C_t(\omega)\cap \left\{(h,\tilde{h})\in \mathbb{R}^d\times L^2(\nu(\td x))\left||h|^2+\IR |\tilde{h}(x)|^2 \nu(\td x)\leq n^2\right.\right\}.$$
Since the sets $C^n_t(\omega)$, $n\in\mathbb N$, are closed in $\mathbb{R}^d\times L^2(\nu(\td x))$, 
we have that also $C_t(\omega)$ is closed.

As $g^n$, $n\in\mathbb N$, are convex positively homogeneous driver functions it follows by Lemma \ref{lemmaint} 
that $0\in\mathrm{int}(C^n)$. As $C^n\subset C$ we have thus that $0\in\mathrm{int}(C)$.

Finally, to show that $C$ satisfies the desired representation (\ref{cr})--\eqref{S} we note that $D_0(X)$ is equal to
\begin{align*}
\sup_{n\in\mathbb N} D^{(n)}_0(X) &= \sup_{n\in\mathbb N}\sup_{\{(H,\tilde H)|(H_s,\tilde{H}_s)\in C^n_s, s\in[0,T]\}}\E{\int_0^T  
\left(H_s^{\intercal} H_s^X+\IR \tilde{H}_s(x)\tilde{H}^X_s(x))\nu(\td x)\right)\td s}\\
&=\sup_{\{(H,\tilde H)|(H_s,\tilde{H}_s)\in C_s, s\in[0,T]\}}\E{\int_0^T  \left(H_s^{\intercal} H_s^X+\IR \tilde{H}_s(x)\tilde{H}^X_s(x)\nu(\td x)\right)\td s}\\
&= \sup_{\{\xi\in\mathcal Q|(H^\xi_s,\tilde{H}^\xi_s)\in C_s, s\in[0,T]\}}\E{\xi X}, 
\end{align*}
where in the first and second line the suprema are taken over pairs $(H,\tilde H)\in L_d^2(\mathcal{P},\td\mathbb P\times \td t)\times
L^2(\mathcal{P}\times \mathcal{B}(\R^k\setminus\{0\} ),\td\mathbb P\times \td t \times \nu(\td x))$. 
This yields (\ref{cr})--\eqref{S} for $s=0$, and hence for all $s\in[0,T]$ by Remark~\ref{rem1}(ii).
Thus, the implication `$\Rightarrow$' is shown, and the proof is complete. \qed

\section{Dynamic mean-deviation portfolio optimisation}\label{secap}
We turn  next to the stochastic optimisation problem of identifying a dynamic portfolio allocation strategy
that maximizes the sum of the expected return and a penalty for its riskyness given in terms of a dynamic deviation measure of the final wealth achieved under this allocation strategy.
Throughout this section we impose the following conditions: 
\begin{As}\label{asm}
{\bf (i)} The L\'{e}vy measure $\nu$ is such that
$\nu(\{x\in\R^k\backslash\{0\}: \min_{i=1,\ldots, k} x_i \leq -1\}) = 0$,
and 
\begin{equation}\label{nuc}
\nu_2 := \int_{\R^k\backslash\{0\}} |x|^2 \nu(\td x) < \infty. 
\end{equation}
{\bf (ii)} $D$ is a $g$-deviation measure with non-random, time-independent driver 
$\hat g: \R^d\times L^2(\nu(\td x))\to \R_+$. 
\end{As}
Under \eqref{nuc}, $L=(L^1_t, \ldots, L^k_t)^{\intercal}_{t\in[0,T]}$ with $L^j_t 
= \int_{[0,t]\times\R^k\backslash\{0\}} x_j \tilde N(\td s \times\td x)$, $j=1,\ldots, k$, where $x_j$ 
is the $j$th coordinate of $x\in\mathbb R^k$, is a vector of pure-jump $(\mathcal F_t)$-martingales.

The financial market that we consider consists of a bank-account that pays interest at a fixed rate $r\ge0$
and $n$ risky stocks (with $1\leq n\leq \min\{d,k\}$) with price processes $S^i=(S^i_t)_{t\in[0,T]}$, 
$i=1,\ldots, n$, satisfying the SDEs given by
\begin{eqnarray} \label{Si1}
&& \frac{\td S_t^i}{S^i_{t-}} = \mu_i\,\td t + \sum_{j=1}^d\sigma_{ij}\,\td W^j_t + \sum_{j=1}^k\rho_{ij}\,\td  L^j_t, \q t\in(0,T], 
\end{eqnarray}
where $S^i_0= s_i \in \mathbb R_+\backslash\{0\}$, $\mu_i\in\R$, $\sigma_{ij}\in\R_+$ and $\rho_{ij}\in\R_+$ such that $\sum_{j=1}^{k}\rho_{ij} \leq 1$ denote the rates of appreciation, the volatilities 
and the jump-sensitivities. By $\pi=(\pi^1, \ldots, \pi^n)^{\intercal}$ we denote the dynamic allocation process
that indicates the fraction  of the total wealth that is invested in the stocks $1,\ldots, n$ 
(that is, if $X^{\pi}(t-)$ denotes the wealth just before time $t$, $\pi_i(t)X^{\pi}(t-)$ is 
the cash amount invested in stock $i$ at time $t$ under allocation strategy $\pi$). We adopt the standard frictionless setting 
(no transaction costs, infinitely divisible 
stocks, continuous trading, {\em etc.}) and restrict to the case that short-sales and borrowing are not permitted, 
by only considering allocation processes $\pi=(\pi_t)_{t\in[0,T]}$ that take values in the set 
$$\mathcal B = \left\{x\in\R^{1\times n}: \min_{i=1,\ldots, n} x_i \ge 0, \sum_{i=1}^n x_i \leq 1\right\}.$$ 
Such an allocation process $\pi$ is said to be {\em admissible} if (i) $\pi$ is predictable, 
(ii) the associated wealth process $X^{\pi}$ is non-negative 
(that is, $X^{\pi}$ satisfies the insolvency constraint $\inf_{t\in[0,T]} X_t^\pi \ge 0$) 
and (iii) $\pi$ is a self-financing portfolio such that $X^{\pi}$ satisfies the SDE 
(with $\mu=(\mu_1,\ldots, \mu_n)^{\intercal}$, $\Sigma=(\sigma_{ij})\in\R^{n\times d}$ 
and $R = (\rho_{ij})\in\R^{n\times k}$) 
given by
\begin{eqnarray}\label{Xpi}
&& \frac{\td X_t^{\pi}}{X_{t-}^{\pi}} =  [r  + (\mu - r \mathbf{1})^\intercal\pi_t]\,\td t + \pi_t^\intercal\Sigma\, \td W_t 
+ \pi_t^\intercal R\, \td L_t
,\q t\in(0,T],
\end{eqnarray}
with initial wealth $X^{\pi}_0=x \in\mathbb R_+\backslash\{0\}$, where $\mathbf{1}\in\R^{n\times 1}$ denotes 
the column vector of ones. 
We denote by $\Pi$ the collection of admissible allocation strategies and let 
$\gamma>0$ denote a risk-aversion parameter. 
To a given allocation strategy $\pi\in\Pi$ we associate the following dynamic performance criterion:
\begin{equation}\label{mean-dev2}
J^{\pi}_t := E[X_T^{\pi}|\F_t]-\gamma D_{t}(X_T^{\pi}),  \q t\in[0,T].
\end{equation}
Due to the fact that, unlike the conditional expectation, $D_{t}(X)$ is a non-linear function 
of $X$, the Dynamic Programming Principle is not satisfied for this objective.
There is a growing literature exploring alternative solution approaches 
to dynamic optimisation problems for which the Dynamic Programming Principle is not applicable.
One alternative dynamic solution concept is that of subgame-perfect Nash equilibrium---%
in such a game-theoretic approach  the problem~\eqref{mean-dev2} 
may informally be seen as a (non-cooperative) game with infinitely many players, one for each time $t$, 
which may be interpreted in terms of  the changing preferences of one person over time; see 
Ekeland and Pirvu (2008) and Bj\"{o}rk and Murgoci (2010) for background, and see Basak and Chabakauri (2010), 
Bj\"{o}rk and Murgoci (2010), Wang and Forsyth (2011), Czichowsky (2013), 
Bj\"{o}rk {\em et al.} (2014), Bensoussan {\em et al.} (2014), and references therein, for studies of dynamic mean-variance portfolio optimisation problems. 
Following Ekeland and Pirvu (2008) and Bj\"{o}rk and Murgoci (2010) we have the following formalisation of 
this equilibrium solution concept in our setting:
\begin{definition}\label{defeq} {\bf (i)} An allocation strategy $\pi^*\in\Pi$ is an {\em equilibrium policy} for 
the dynamic mean-deviation problem with objective \eqref{mean-dev2} if 
\begin{equation}
\liminf_{h\searrow 0} \frac{J_t^{\pi^*} - J_t^{\pi(h)}}{h} \ge 0  
\end{equation}
for any $t\in[0,T)$ and any policy $\pi(h)\in \Pi$ satisfying, for some $\pi\in\Pi$, 
$$\pi(h)_s = \pi_s I_{[t,t+h)}(s) + \pi^*_sI_{[t+h,T]}(s),\quad s\in[t,T].$$

\noindent{\bf (ii)} An equilibrium policy $\pi^*$ is of {\em feedback type} if, for some 
{\em feedback function} $\pi_*:[0,T]\times\R_+\to\mathcal B$ such that \eqref{Xpi} with $\pi_{t}$ replaced 
by $\pi_*(t,X_{t-})$ has a unique solution $X^*=(X^*_t)_{t\in[0,T]}$, we have 
$$\pi^*_{t} = \pi_*(t, X^*_{t-}),\quad t\in[0,T],$$ 
with $X^*_{0-}=X^*_{0}$.
\end{definition}
For a given equilibrium policy $\pi^*=(\pi^*_t)_{t\in[0,T]}$ of {feedback type} we have 
by the Markov property that $J^{\pi^*}_t = V(t, X^{*}_t)$ and $\E{X^{\pi^*}_T|\F_t} = h(t,X^{*}_t)$, $t\in[0,T]$, 
for some functions $V:[0,T]\times\R_+\to\R_+$ and $h:[0,T]\times\R_+\to\R_+$.
Furthermore, if $h$ is sufficiently regular ({\em e.g.,} $h\in C^{1,2}([0,T]\times\R_+,\R_+)$ 
and $h'\equiv \frac{\partial h}{\partial x}$ 
is bounded) we find by an application of It\^{o}'s lemma  
that the representing pair of $X^{*}_T$ is given by
\begin{eqnarray*}
&& H^{X^{*}_T}_s = a^{*,h}(s,X^{*}_{s-}),\quad 
\tilde H^{X^{*}_T}_s(y) = b^{*,h}(s,X^{*}_{s-},y),\quad s\in[0,T],\ y\in\mathbb R^k\backslash\{0\},\ \text{with}\\
&& a^{*,h}(s,x) := h'(s,x)\,x\, \pi_{*}(s,x)^\intercal\Sigma, \quad
b^{*,h}(s,x,y) := h(s,x + \,x \pi_{*}(s,x)^\intercal R y) - h(s,x),
\end{eqnarray*}
so that $D_t(X^{*}_T)$, $t\in[0,T]$, takes the form $D_t(X^{*}_T) = \tilde D_{t,X^{*}_t}(X^{*}_T)$, 
where 
\begin{equation}
\tilde D_{t,x}(X^{*}_T) = \mathbb E_{t,x}\left[\int_t^T \hat g(a^{*,h}(s,X^{*}_{s-}), b^{*,h}(s,X^{*}_{s-},\,\cdot\,))\td s\right], \q 
(t,x)\in[0,T]\times\mathbb R_+,
\end{equation}
with $\mathbb E_{t,x}[\,\cdot\,] = \E{\cdot|X^{*}_t = x}$.
To any vector $\pi\in\mathcal B$ we associate the operators 
$\mathcal L^{\pi}: f\mapsto \mathcal L^{\pi}f$ and 
$\mathcal G^{\pi}: f\mapsto \mathcal G^{\pi}f $ 
that map $C^{0,2}([0,T]\times\mathbb R_+,\mathbb R)$ to $C^{0,0}(\mathbb R_+,\mathbb R)$ 
and are given by 
\begin{eqnarray}\label{Lp}
\!\!\!\!\!\!\!\!\!&&\mathcal L^{\pi}f(t,x) = \mu_{\pi} xf'(t,x) + \mbox{$\frac{\sigma_\pi^2}{2}$} x^2 f''(t,x) 
+ \int_{\R^k\backslash\{0\}}[f(t,x + x\pi^\intercal R y) - f(t,x) -  x\pi^\intercal R yf'(t,x)]\nu(\td y),\qquad\\
\!\!\!\!\!\!\!\!\!&&\mathcal G^{\pi}f(t,x) = \hat g(xf'(t,x)\, \pi^\intercal\Sigma, \delta_{x\pi^{\intercal} R I} f(t,x)),
\label{Gp}
\end{eqnarray}
for  $(t,x)\in[0,T]\times\mathbb R_+$, where $\delta_y f:\R_+\to\R$ and
$I:\mathbb R^{k\times1} \to \mathbb R^{k\times1}$ 
are given by 
$$\delta_y f(x) = f(t,y+x) - f(t,x),\q I(z)=z,\q 
 z\in\R^{k\times 1}, x\in\R_+, y\in\R,
$$
and where  
$$\mu_\pi = r + (\mu-r\mathbf 1)^{\intercal}\pi,\qquad \sigma_\pi^2 = \pi^{\intercal}\Sigma\Sigma^{\intercal}\pi, 
\qquad \pi\in\mathcal B.$$

\noindent Given the form of the objective and Definition~\ref{defeq}
we are led to consider the {\em extended Hamilton-Jacobi-Bellman equation} for a triplet 
$(\pi_*, V,h)$ of a feedback function $\pi_*$, the corresponding value function $V$ 
and auxiliary function $h$ given by (denoting $\dot{V} = \frac{\partial V}{\partial t}$):
\begin{eqnarray}\label{H1}
&& \dot{V}(t,x) + \sup_{\pi\in\mathcal B}\{\mathcal L^{\pi}V(t,x) - 
\gamma \mathcal G^{\pi}h(t,x)\} = 0,\quad \,   (t,x)\in [0,T)\times\R_+\backslash\{0\},\\
\label{HH1}
&&\dot{h}(t,x) + \mathcal L^{\pi_*(t,x)}h(t,x) = 0,\qquad \qquad\qquad\quad \ \ \  (t,x)\in [0,T)\times\R_+\backslash\{0\},\\
&& V(T,x) = h(T,x) = x, \qquad\qquad\!  x\in\R_+, \label{HH2}\\
&& V(t,0) = h(t,0) = 0, \qquad\qquad\q\!\!  t\in[0,T], \label{H2}
\end{eqnarray} 
where, for any $t\in[0,T]$ and $x\in\R_+$, $\pi_*(t,x)$ is a maximiser of the supermum in \eqref{H1}
(note that the continuity of $\mathcal L^{\pi}V(t,x)$ and $\mathcal G^{\pi}h(t,x)$ 
in $\pi$ for each fixed $t\in[0,T]$ and $x\in\R_+\backslash\{0\}$ in conjunction with the compactness of $\mathcal B$ 
guarantees that the maximum in \eqref{H1} is attained).

We have the following verification result:
\begin{theorem}\label{thm:V} Let $(\pi_*,h,V)$ be a triplet satisfying the extended HJB equation~\eqref{H1}--\eqref{H2}, 
let $X^*$ be the unique solution of \eqref{Xpi} with $\pi_t$ replaced by $\pi_*(t,X_{t-})$ and define 
$\pi_t^* = \pi_*(t, X^{*}_{t-})$, $t\in[0,T]$.
Assume $h,V\in C^{1,2}([0,T]\times\R_+,\R)$ with $h', V'$ bounded and that $\pi^*=(\pi^*_t)_{t\in[0,T]}\in\Pi$.
Then $\pi^*$ is an equilibrium policy of feedback type and $h$ and $V$ are given by
$V(t,x) = \mathbb E_{t,x}[X_T^{\pi^*}] - \gamma \tilde D_{t,x}(X_T^{\pi^*})$ 
and $h(t,x) = \mathbb E_{t,x}[X_T^{\pi^*}]$ for $(t,x)\in[0,T]\times\R_+$.
\end{theorem}
\begin{proof}{}
We first verify the stochastic representations.
Let $\pi=(\pi_s)_{s\in[0,T]}\in\Pi$, $t\in[0,T)$ and $\tau\in(0,T-t)$ be given 
and denote $\Xi^{\pi,V,h}(s,x) := (\dot{V} + \mathcal L^{\pi_s}V - \gamma \mathcal G^{\pi_s}h)(s,x)$, 
$a^{\pi,V}(s,x) :=  V'(s,x) x \pi_s^\intercal\Sigma$ 
and $b^{\pi,V}(s,x,y) = V(s,x + x\pi_s^\intercal Ry) - V(s,x)$. 
An application of It\^o's lemma to $V(t+\tau, X^{\pi}_{t+\tau})$ shows that
\begin{multline}\label{VIto}
V(t+\tau, X^{\pi}_{t+\tau}) - V(t, X^{\pi}_{t})   - \gamma 
\int_t^{t+\tau}\mathcal G^{\pi_s}h(s,X_{s-}^{\pi})\td s = \int_t^{t+\tau} \Xi^{\pi,V,h}(s,X^{\pi}_{s-}) \td s\\
+ \int_t^{t+\tau} a^{\pi,V}(s,X^{\pi}_{s-}) \td W_s  + 
\int_{(t,{t+\tau}]\times\R^k\backslash\{0\}} b^{\pi,V}(s,X^{\pi}_{s-},y) \tilde N (\td s \times \td y).
\end{multline}
Similarly, it follows $h(t+\tau, X^{\pi}_{t+\tau})$ 
satisfies \eqref{VIto} with $V$, $\mathcal G^{\pi_s}h$ and $\Xi^{\pi,V,h}$ replaced by $h$, 0 and 
$\Xi^{\pi,h,0}$, respectively. In particular, choosing $\pi$ equal to $\pi^*$,  and taking expectations, 
the three terms on the right-hand side of \eqref{VIto} vanish in view of \eqref{H1}, the fact that $\pi_*(t,x)$ is a maximiser in \eqref{H1} and as the stochastic integrals are martingales (in view of the boundedness of $V', h'$ and $\mathcal B$). 
Then, letting $\tau\nearrow T-t$ and using the boundary conditions \eqref{HH2},  
we obtain the asserted stochastic representations of $h$ and $V$. 

Next we turn to the proof that $\pi^*$ is an equilibrium solution. By an application of the tower-property of conditional expectation and (D6) we have
for any $\pi\in\Pi$, $t\in[0,T)$ and $\tau\in(0,T-t)$ 
\begin{eqnarray}\nonumber
J_t^{\pi} &=& \E{\E{X^{\pi}_T|\mathcal F_{t+\tau}}|\mathcal F_t} 
- \gamma \E{D_{t+\tau}(X_T^{\pi})|\mathcal F_t}
-\gamma D_t(\E{X^{\pi}_T|\mathcal F_{t+\tau}})\\
&=& \E{J_{t+\tau}^\pi|\F_t} - \gamma D_t(\E{X^{\pi}_T|\mathcal F_{t+\tau}}).
\label{Jrec}
\end{eqnarray}

Fixing $(\epsilon_n)_n$, $\epsilon_n\searrow 0$, and strategies $\pi_n:=\pi(\epsilon_n)\in\Pi$ as in 
Definition~\ref{defeq} (with $\pi^*$ as asserted in the theorem) and noting that the Markov property (which is in force 
as $\pi^*$ is a feedback strategy)
implies  
\begin{equation}\label{hVmp}
J_{t+\epsilon_n}^{\pi_n} = V(t+\epsilon_n, X^{\pi_n}_{t+\epsilon_n}),\q
\E{X^{\pi_n}_T|\mathcal F_{t+\epsilon_n}} = h(t+\epsilon_n, X^{\pi_n}_{t+\epsilon_n}),
\end{equation}
and that \eqref{hVmp} remains valid with $\pi_n$ replaced by $\pi^*$,
we have from \eqref{VIto} and \eqref{Jrec} and the fact 
$D_t(h(t+\epsilon_n,X_{t+\epsilon_n}^{\pi_n} )) = \E{\left.\int_t^{t+\epsilon_n}\hat g(a_s^{\pi_n,h}, b_s^{\pi_n,h})\td s\right|\mathcal F_t}$ 
that
\begin{equation}
J_t^{\pi^*} - J_t^{\pi_n} = \E{\left.\int_t^{t+\epsilon_n}
[\Xi^{\pi^*,V,h}(s,X^{\pi^*}_{s-}) - \Xi^{\pi_n,V,h}(s,X^{\pi_n}_{s-})] 
\td s\right|\F_t}.
\end{equation}
Since $\Xi^{\pi_*(s,x),V,h}(s,x)=0$ and $\Xi^{\pi_n,V,h}(s,x)\leq 0$ for $s\in[0,T]$, $x\in\R_+$, 
(by \eqref{H1} and the fact that $\pi_*(t,x)$ is the maximiser in \eqref{H1}) we have 
$
\liminf_{n\to\infty}(J_t^{\pi^*} - J_t^{\pi_n})/\epsilon_n 
\ge 0,  
$
and the proof is complete.
\end{proof}

We next identify an explicit equilibrium policy for the mean-deviation portfolio optimisation problem, 
under the following regularity assumption on $\Sigma$, $R$ and $\hat g$, assumed to be in force in the sequel:

\begin{As}\label{AS2}
For some countable set $A$ and any $a\in[0,\gamma^{-1}]\backslash A$, the function $T_a:\mathcal B\to\mathbb R$
given by
\begin{equation}\label{Ta}
T_a(c) := a\,(\mu-r\mathbf 1)^\intercal c - \hat g(c^\intercal\Sigma,c^\intercal R I),\quad c\in\mathcal B,
\end{equation}
achieves its maximum over $\partial B$ at a unique $c^*\in\partial\mathcal B$.\footnote{$\partial\mathcal B$ denotes the boundary of $\mathcal B$, that is, $\partial B=\mathrm{cl}(\mathcal B)\backslash \mathrm{int}(\mathcal B)$ 
where  $\mathrm{cl}(\mathcal B)$ and  $\mathrm{int}(\mathcal B)$ denote the closure and interior of $\mathcal B$.} 
\end{As}

\noindent To define the optimal policy we deploy the following auxiliary result:

\begin{Lemma}\label{lem:faux}
For any $f:[0,T]\to\mathcal B$ denote by 
$A_f, d_f, b_f, F_f: [0,T]\to\mathbb R$ the functions 
given by
\begin{eqnarray}\label{bf}
b_f(t) &:=& \exp\left(\int_t^T\{r + (\mu-r\mathbf 1)^{\intercal} f(s)\}\td s\right),\\ 
\label{df}
d_f(t) &:=& b_f(t) \int_t^T \hat g(f(s)^\intercal\Sigma, f(s)^\intercal R I)\td s,\\
A_f(t) &:=& \gamma^{-1} - (b_f(t))^{-1} d_f(t),\\
F_f(t) &:=& A_{C_f}(t), \q\text{with}\\ \label{Cf}
C_f(t) &:=& 
\begin{cases}
\mathrm{arg}\,\sup_{c\in\partial\mathcal B}\,\left\{T_{f(t)}(c)\right\}, & \text{if $f(t)\notin A$},\\
\vspace{-0.4cm}\\
\mathrm{Centroid}(\mathrm{arg}\,\sup_{c\in\partial\mathcal B}\,\left\{T_{f(t)}(c)\right\}), & \text{if $f(t)\in A$},
\end{cases}
\end{eqnarray}
where for any Borel set $A'\subset\mathbb R^d$, $\mathrm{Centroid}(A')$ is equal to the mean of $U\sim\mathrm{Unif}(A')$. 
Then there exists a continuous non-decreasing function $a^*:[0,T]\to\mathbb R_+$ such that 
$a^* = F_{a^*}$.
\end{Lemma} 
The proof of Lemma~\ref{lem:faux} is provided below.
With this result in hand we identify an equilibrium policy as follows:
\begin{theorem}\label{thm:opt} With $T_a(c)$ and $a^*$ given in \eqref{Ta} and
in Lemma~\ref{lem:faux}, we let $s(a):= \sup_{c\in\partial\mathcal B}T_a(c)$,\,%
$a_-:=\sup\{a\in[0,\gamma^{-1}]: s(a)\leq 0\}$, 
and $t^* := \sup\{t\in[0,T]: a^*(t) \leq a_-\}$ (where $\sup\emptyset:=-\infty$).

\noindent{\bf (i)} If $s(1/\gamma)\leq 0$ then $\pi^* \equiv 0$ with value-function given by $V(t,x)=x\,\exp(r(T-t))$ 
for $(t,x)\in[0,T]\times\mathbb R_+$.

\noindent{\bf (ii)} If $s(1/\gamma)>0$ define the function $C^*:[0,T]\to\mathcal B$ by 
$$
C^*(t) = 
\begin{cases}
C_{a^*(t)}, & \text{if $t\in[t^*\vee 0,1]$,}\\
0,  &  \text{otherwise,}
\end{cases}
$$
where $C_{a^*(t)}$ is given in \eqref{Cf} with $f=a^*$.
Then $\pi^* = C^*$  
is an equilibrium policy with value function given by
$V(t,x) = x(b_{C^*}(t) - \gamma d_{C^*}(t))$ 
for $(t,x)\in[0,T]\times\mathbb R_+$, where 
$b_{C^*}$ and $d_{C^*}$ are given in \eqref{bf} and \eqref{df} with $f=C^*$.
\end{theorem}
\begin{Remark}
Under the equilibrium policy $\pi^*$ given in Theorem~\ref{thm:opt} it is optimal to invest in the $n$ stocks 
according to the proportions $C^*=(C^*_1, \ldots, C^*_n)$ of the current wealth, 
which are non-random functions of $t$ only. 
Hence, it is optimal to invest at time $t$ an amount 
$X^{\pi^*}(t-)C_i^*(t)$ in stock $i$, $i=1, \ldots, n$. 
\end{Remark}
\begin{proof}{\ of Theorem~\ref{thm:opt}}
The proof consists in verifying that the triplet $(\pi_*,V,h)$, with $\pi_*$ 
and $V$ as stated and
with $h:[0,T]\times\mathbb R_+\to\mathbb R$ given by  $h(t,x)=x\, b_{C^*}(t)$, 
satisfies the extended HJB equation~\eqref{H1}--\eqref{H2}; the assertions then follow 
by an application of Theorem~\ref{thm:V}.

\noindent{\bf (i)} Once we verify that the supremum in \eqref{H1} is attained at $\pi_*\equiv 0$ 
it is easily checked that $V$ and $h$ are equal and satisfy \eqref{H1}--\eqref{H2}, using that $g$ is positively homogeneous. 
To see that the former is the case note that the left-hand side of \eqref{H1} is equal to 
$x\, \exp(r(T-t))\,[-r + \gamma\sup_{c\in\mathcal B}T_{1/\gamma}(c)]$; 
since $s(1/\gamma)\leq 0$, the latter supremum is zero and it is attained at $c=0$ (as $T_{1/\gamma}(0)=0$).

\noindent{\bf (ii)} Assume for the moment that the supremum in \eqref{H1}is attained at $\pi^*$. 
Then the positive homogeneity of $g$ and the fact (which is straightforward to verify) 
that functions $b_{C^*}$ and $d_{C^*}$ satisfy the 
system of equations
\begin{eqnarray*}
&& \dot{b} + (r + \mu_{C^*}) b = 0, \q t\in[0,T),\q b(T) = 1,\\
&& \dot{d} + (r + \mu_{C^*}) d 
+ b \hat g((C^{*})^\intercal\Sigma, (C^{*})^\intercal R I) = 0, \q t\in[0,T),\q d(T) = 0,
\end{eqnarray*}
where as before $I:\mathbb R^{k\times 1}\to\mathbb R^{k\times 1}$ is given by $I(y)=y$, imply that 
$h$ and $V$ satisfy \eqref{H1}--\eqref{H2}. 

Next we verify that the supremum in \eqref{H1} 
is attained at $\pi_*$. Inserting the forms of $h$ and $V$ 
and using that $\gamma\, \inf_{t\in[0,T]} b_{C^*}(t)>0$
we have for any $t\in[0,T]$ that 
\begin{eqnarray}\nonumber
\text{arg}\,\sup_{\pi\in\mathcal B}\{\mathcal L^{\pi}V(t,x) - 
\gamma \mathcal G^{\pi}h(t,x)\}
&=& \text{arg}\,\sup_{\pi\in\mathcal B}\{\mu_{\pi}(b_{C^*}(t) - \gamma d_{C^*}(t)) 
-\gamma b_{C^*}(t) \hat g(\pi^{\intercal}\Sigma, \pi^{\intercal}R\, I)\}\\
&=& \text{arg}\,\sup_{\pi\in\mathcal B}\{\mu_{\pi} A_{C^*}(t) 
- \hat g(\pi^{\intercal}\Sigma, \pi^{\intercal}R\, I)\}.
\label{arg}
\end{eqnarray}
If $t\leq t^*$, then  $A_{C^*}(t) = a_-$ so that $s(A_{C^*}(t))\leq 0$ and $0$ is included in the argsup in \eqref{arg}, 
while if $t>t^*$, then $A_{C^*}(t)> a_-$ and we have that $s(A_{C^*}(t)) = \sup_{\pi\in\mathcal B}\{(\mu_{\pi}-r) A_{C^*}(t) 
- \hat g(\pi^{\intercal}\Sigma, \pi^{\intercal}R\, I)\}>0$ is attained at 
$\pi=C_{A_{C^*}}(t) = C_{a^*(t)} = C^*(t)$.
\end{proof}
\begin{proof}{\ of Lemma~\ref{lem:faux}}
The proof relies on an application of Schauder's fixed point theorem\footnote{see {\em e.g.} Theorem 1.C in 
Zeidler (1995)} to the map 
$F: \mathbb A\to C([0,T],\mathbb R)$ given by $f\mapsto F_f$, where $\mathbb A$ denotes the set 
of continuous functions $f\in C([0,T],\mathbb R)$ 
that are such that (a) $f(T) = \gamma^{-1}$ and (b) for all $s,t\in[0,T]$ with $s\leq t$ we have 
$f(t)-f(s)\in[\chi_-(t-s), \chi_+ (t-s)]$ where
\begin{eqnarray*}
\chi_+ := \sup\{\hat g(c^\intercal\Sigma, c^\intercal R I) : c\in\partial\mathcal B\},\quad 
\chi_- := \inf\{\hat g(c^\intercal\Sigma, c^\intercal R I) : c\in\partial\mathcal B\}. 
\end{eqnarray*}
We note that both $\chi_+$ and $\chi_-$ are strictly positive, by positivity of the driver function $\hat g$.
It is straightforward to verify that $F$ maps $\mathbb A$ to $\mathbb A$ 
and that  the set $\mathbb A$ is a non-empty, closed, bounded and convex subset of $C([0,T],\mathbb R)$. 
Since $F$ is compact (as we prove below), Schauder's fixed point theorem yields that there exists 
an element $a^*\in\mathbb A$ such that $a^* = F_{a^*}$.

We next prove that $F$ is compact by showing that {\bf (i)} $F$ is continuous (with respect to the supremum norm on $[0,T]$) 
and {\bf (ii)} the set $F(\mathbb A) = \{F_f: f\in\mathbb A\}$ is relatively compact in $C([0,T],\mathbb R)$. 

\noindent{\bf (i)} Let $(f_n)_n\subset\mathbb A$ converge to $f\in \mathbb A$ in the supremum-norm. Then we have that 
$T_{f_n(t)}(c)\to T_{f(t)}(c)$ as $n\to\infty$ uniformly in $t\in[0,T]$ for any $c\in\partial\mathcal B$, and 
$\sup_{c\in\partial\mathcal B}T_{f_n(t)}(c)\to \sup_{c\in\partial\mathcal B}T_{f(t)}(c)$ for any $t\in[0,T]$. 
As $(f_n)_n$ and $f$ are strictly monotone increasing and  Assumption~\ref{AS2} is in force, we have for all but 
countably many $t$ that $T_{f_n(t)}(c)$ and $T_{f(t)}(c)$ attain their maxima over $\partial B$ at unique $c$. 
Thus, it follows that $\mathrm{arg}\,\sup_{c\in\partial\mathcal B}T_{f_n(t)}(c)\to \mathrm{arg}\,\sup_{c\in\partial\mathcal B}T_{f(t)}(c)$, 
for a.e. $t\in[0,T]$. Hence, by the dominated convergence theorem $F_{f_n}(t) = A_{C_{f_n}}(t)\to A_{C_f}(t)=F_f(t)$ for any $t\in[0,T]$. Since the functions $A_{C_{f_n}}$ and $A_{C_f}$ are non-decreasing, the convergence $F_{f_n}\to F_f$ holds 
in the supremum norm.

\noindent{\bf (ii)}  
Using the boundedness of $\mathcal B$ and the continuity of $\hat g$ it is straightforward to verify that
the collection of functions $F(\mathbb A)$ is equi-continuous. 
Hence we have by an application of 
the Arzela-Ascoli theorem\footnote{see {\em e.g.} p.35 in Zeidler (1995)} that for any sequence 
$(A^{(n)})_n\subset F(\mathbb A)$  there exists a continuous function $A^*:[0,T]\to\mathbb R$ such that, along a subsequence 
$(n_k)$,  $(A^{(n_k)})_k$ converges uniformly to $A^*$, hence establishing that $F(\mathbb A)$ 
is relatively compact. 
\end{proof}
\begin{Example}{\bf (i)} For driver function $\hat g=g_1$ (given in Example~\ref{D3} with $\lambda=1$)  
and for $a\in\mathbb R_+$ we have that $T_a(c)$ in \eqref{Ta} is given by
\begin{eqnarray*}
T_a(c) = a\,(\mu - r\mathbf 1)^\intercal c - \sqrt{c^{\intercal}\Sigma\Sigma^{\intercal} c + c^{\intercal}RR^{\intercal} c\nu_2}.
\end{eqnarray*}
If $\Sigma\Sigma^{\intercal} + RR^{\intercal}\nu_2$ is invertible,
then it is straightforward to verify that Assumption~\ref{AS2} is satisfied. 

\noindent{\bf (ii)} Let us identify explicitly the equilibrium portfolio allocation strategy given in Theorem~\ref{thm:opt}
in the case the driver function $\hat g$ is as in part (i) and we have $2$ risky assets ($n=2$), 
whose dynamics we suppose are given by \eqref{Si1} with
$d=k=2$, $\mu_1>\mu_2>r$, $r\ge 0$ and $s_{12}:= (\Sigma^2 + R^2\nu_2)_{12} < 0$. 
In terms of $s_i^2 := (\Sigma^2 + R^2\nu_2)_{ii}$, $i=1,2$, 
let us denote 
\begin{eqnarray*}
&&d_+ := s_1^2 + s_2^2 - 2 s_{12},\quad e_+ := s_{12} - s_2^2,\\ 
&&c_+(a) := - \frac{e_+}{d_+} + \sqrt{\left(\frac{e_+}{d_+}\right)^2 - \eta(a)},\quad
\eta(a) := \frac{a^2(\mu_1-\mu_2)^2 s_2^2 - e_+^2}{d_+(a^2(\mu_1-\mu_2)^2 - d_+)},
\end{eqnarray*}
for $a\in[0,\sqrt{d_+}/(\mu_1-\mu_2))$. By convexity of $g$ it follows that 
the supremum of $\tilde T(c):= T_a((c,1-c)) = a(\mu_2-r) + a(\mu_1-\mu_2) c - \sqrt{d_+ c^2 + 2e_+ c + s^2_2}$
over $c\in\mathbb R$ is attained at the $c$ satisfying 
$\tilde T'(c)=0 \Leftrightarrow c = c_+(a)$ and we have 
$$\tilde T'(1)>0 \Leftrightarrow a > a_+ := 
\frac{1}{\mu_1 - \mu_2}\left(\frac{s_1^2 - s_{12}}{\sqrt{s_1^2}}\right).$$ 
As a consequence, the equilibrium allocation strategy $\pi^*=(\pi^*_t)_{t\in[0,T]}$ in 
Theorem~\ref{thm:opt} is given as follows:
\begin{eqnarray*}
\pi^*_t = C^*(t) = 
\begin{cases}
(1,0), & \text{if }\, a_*(t) > a_- \vee a_+,\\
(c_+(a_*(t)), 1 - c_+(a_*(t))), & \text{if }\, a_- < a_*(t) \leq a_-\vee 
a_+,\\ 
(0,0), & \text{if }\, a_*(t) \leq a_-,
\end{cases}
\end{eqnarray*}
where $a_-$ and $a_*(t)$ are as in Theorem~\ref{thm:opt}.  
Hence, if the risk-aversion parameter $\gamma$ is sufficiently small 
and/or $t$ is sufficiently close to the horizon $T$ the equilibrium strategy is to be fully invested in 
risky asset 1, which has the highest expected return; at times $t$ further away from the horizon or for higher 
risk-aversion parameter, the dynamic deviation penalty term starts to play a more important role and the policy 
is to invest part of the wealth into asset 2, 
while, if $\gamma$ is sufficiently large or $t$ is sufficiently small, the equilibrium strategy 
is to invest all the wealth in the bank account.

\noindent{\bf (iii)} Restricting next to the case of a single risky asset ($n=1$) with 
$d=k=1$, $\mu:=\mu_1>r$, we find by a direct calculation that the value function $V$ in Theorem~\ref{thm:opt}
and he auxiliary function $h$
are explicitly given in terms of 
$$t^* = \left(T + \frac{1}{\mu - r} - \frac{1}{\gamma \sqrt{\Sigma^2 + R^2\nu_2}}\right) \wedge T$$
by $V(t,x) = V(t^*\wedge T, x\,\exp\{r(t^*\wedge T - t)\})$ and 
$h(t,x) = h(t^*\wedge T, x\,\exp\{r(t^*\wedge T - t)\})$
for $t\in[0, t^* \wedge T)$ and
\begin{eqnarray*}
&&V(t,x) = h(t,x)[1 - (T-t)\gamma\sqrt{\Sigma^2 + R^2\nu_2}], \q h(t,x) = x \exp\{\mu(T-t)\},\q t\in[t^*\wedge T,T],
\end{eqnarray*}
where the equilibrium policy $\pi^*$ is given by
\begin{eqnarray*}
&&\pi^*_t = C^*(t) = \begin{cases}
1, & \text{if $a(t) = \frac{1}{\gamma}\, \frac{1}{1 + (\mu - r)(T-t)} > 
\frac{\sqrt{\Sigma^2 + R^2\nu_2}}{\mu - r} = a_- \Leftrightarrow t\in(t^*\wedge T,T]$,}\\
0, & \text{if $a(t)\leq a_- \Leftrightarrow t\in[0,t^*\wedge T]$}.
\end{cases}
\end{eqnarray*}
To see that $\pi^*$ takes this form we observe that 
$t\leq t^*$ holds precisely if  $\left(\mu -r - \gamma\, \sqrt{\Sigma^2 + R^2\nu_2}\right) 
- (\mu - r)\gamma(T-t)\sqrt{\Sigma^2 + R^2\nu_2} \leq 0$ 
$\Leftrightarrow$ $0 \in \operatornamewithlimits{arg.\,sup}_{\pi\in[0,1]}\{(\mathcal L^{\pi}V)(t,x) - 
\gamma (\mathcal G^{\pi}h)(t,x)\}$, where
$\mathcal L^{\pi}$ and $\mathcal G^{\pi}$ are given in \eqref{Lp} and \eqref{Gp}.
\end{Example}

\noindent{\bf Acknowledgements.} MP acknowledges support in part by EPSRC grant EP/I019111/1. 
MS acknowledges support by NWO VENI 2012.

\end{document}